\documentclass[12pt,reqno,letterpaper]{amsart}
\usepackage{latexsym,amsmath,amssymb, amsthm, mathscinet}
\usepackage{cases, verbatim}

\usepackage{amsfonts,amsmath,amsthm, amssymb}
\usepackage{cases}
\usepackage{enumerate}

\newtheorem{thm}{Theorem}[section]

\newtheorem{defn}[thm]{Definition}

\newtheorem{ex}[thm]{Example}
\newtheorem{lem}[thm]{Lemma}
\newtheorem{cor}[thm]{Corollary}
\newtheorem{prop}[thm]{Proposition}

\numberwithin{equation}{section}

\newcommand{\N}{\mathbb{N}}

\newcommand{\R}{\mathbb{R}}

\newcommand{\T}{\mathbb{T}}

\newcommand{\BUC}{{\rm BUC\,}}

\newcommand{\Lip}{{\rm Lip\,}}

\newcommand{\ep}{\varepsilon}

\newcommand{\ol}{\overline}

\newcommand{\diam}{{\rm diam}\,}

\begin{document}
\title[A PDE formulation of Lyapunov stability]{A PDE formulation of Lyapunov stability for contact-type Hamilton--Jacobi equations}

\author{Panrui Ni}
\address[P. Ni]{
Department 1: Department of Mathematics, Faculty of Science and Engineering, Waseda University, 3-4-1 Okubo, Shinjuku-ku, Tokyo 169-8555, Japan; Department 2: Shanghai Center for Mathematical Sciences, Fudan University, Shanghai 200438, China}
\email{panruini@gmail.com}

\author{Jun Yan}
\address[J. Yan]{
School of Mathematical Sciences, Fudan University, Shanghai 200433, China}
\email{yanjun@fudan.edu.cn}

\makeatletter
\@namedef{subjclassname@2020}{\textup{2020} Mathematics Subject Classification}
\makeatother

\date{\today}

\begin{abstract}
We study the Lyapunov stability of stationary viscosity solutions of contact-type Hamilton--Jacobi equations on a compact manifold. We consider Hamiltonians of the form $H(x,p,u)=G(x,p)+W(x,u)$, where $G$ is continuous, convex and coercive in $p$. We establish sufficient conditions for asymptotic stability and instability in terms of viscosity subsolutions and one-sided derivatives of a perturbed critical value. These criteria do not require the construction of a smooth contact Hamiltonian flow. Under the stability condition, solutions starting sufficiently close to a stationary solution converge to it at an exponential rate. To prove instability, we use the superposition principle for closed measures to construct calibrated curves. We also establish conditions ensuring uniqueness of stationary solutions and global asymptotic stability. Our results also cover examples with nonmonotone dependence on the unknown function, including an equation with an eikonal term $|p|$. Finally, we obtain an $O(\sqrt{\varepsilon})$ convergence rate for a class of stationary Hamilton--Jacobi equations with nonlinear dependence on the unknown function and multiscale spatial dependence.
\end{abstract}

\maketitle



\section{Introduction}

In this paper, we consider the following contact-type Hamilton--Jacobi equation
\begin{equation}\label{ee}\tag{eHJ}
  \left\{
   \begin{aligned}
   &u_t(x,t)+H(x,Du(x,t),u(x,t))=0\quad \text{for}\ (x,t)\in M\times(0,+\infty),
   \\
   &u(x,0)=\varphi(x) \quad\text{for} \ x\in M,
   \end{aligned}
   \right.
\end{equation}
where $M$ is a compact, connected, smooth manifold without boundary, $H:T^*M\times\R\to\R$ is the Hamiltonian, and $\varphi\in C(M)$. Throughout the paper, all solutions are understood in the viscosity sense, and the term “viscosity” will often be omitted. The large-time behavior of $u(x,t)$ can be complicated when the monotonicity of the map $u\mapsto H(x,p,u)$ fails; see \cite{WYZ1,WYZ2}. Such non-monotone dependence in $u$ appears, for example, in certain models of dislocation dynamics; see \cite{IM,MNT}. For $t>0$, let $T^-_t:C(M)\to C(M)$ be the nonlinear operator defined in \eqref{T-} below. Then the map $(x,t)\mapsto T^-_t\varphi(x)$ gives the solution of \eqref{ee}. In this way, $(T^-_t,C(M))$ can be regarded as an infinite-dimensional dynamical system. A fundamental notion in dynamical systems is the Lyapunov stability of fixed points. In \cite{RWY}, the authors investigate the Lyapunov stability of fixed points of $T^-_t$. These fixed points correspond to solutions of
\begin{equation}\label{se}\tag{sHJ}
H(x,Du(x),u(x))=0\quad \text{for }x\in M.
\end{equation}
Existence results for stationary solutions under additional hypotheses can be found in \cite{JMT,WWY1}. We now recall the definition of Lyapunov stability in the present setting.
\begin{defn}\cite[Definition 1.1]{RWY}.
Let $u\in C(M)$ be a solution of \eqref{se}.
\begin{itemize}
\item [(1)] The function $u$ is said to be Lyapunov stable if for any $\ep>0$, there exists $\delta>0$ such that for any $\varphi\in C(M)$ satisfying $\|\varphi-u\|_\infty\leq \delta$, one has
\[\|T^-_t\varphi-u\|_\infty<\ep\quad \text{for all }t>0.\]
Otherwise, $u$ is said to be Lyapunov unstable.

\item [(2)] The function $u$ is said to be asymptotically stable if it is Lyapunov stable and there exists $\delta>0$ such that
\[\lim_{t\to+\infty}\|T^-_t\varphi-u\|_\infty=0\]
for any $\varphi\in C(M)$ with $\|\varphi-u\|_\infty<\delta$. If $\delta$ can be taken as $+\infty$, then $u$ is called globally asymptotically stable.
\end{itemize}
\end{defn}
In \cite{RWY}, the Hamiltonian $H(x,p,u)$ is assumed to be $C^3$ Tonelli, that is, strictly convex and superlinear in $p$. Under conditions formulated in terms of Mather measures associated with contact Hamiltonian systems, sufficient conditions for the Lyapunov stability and instability of fixed points of $T^-_t$ are established. These assumptions arise naturally from a dynamical systems viewpoint. Related results in the smooth contact Hamiltonian setting can be found in \cite{XYZ}. In the present paper, we establish corresponding stability and instability criteria when the momentum dependence of the Hamiltonian is merely continuous, convex and coercive, without assuming differentiability or strict convexity. This lower regularity is obtained under the separated structure below. Our criteria are formulated through viscosity subsolutions and one-sided derivatives of a perturbed critical value, without constructing a smooth contact Hamiltonian flow.
More precisely, we assume that the Hamiltonian takes the form
\begin{equation}\label{H=}
H(x,p,u)=G(x,p)+W(x,u),
\end{equation}
and satisfies the following conditions:
\begin{itemize}
\item[(A1)] The function $G(x,p)$ is continuous, convex in $p$, and \[\lim_{|p|\to+\infty}\inf_{x\in M}G(x,p)=+\infty.\]
\item[(A2)] Both $W(x,u)$ and its partial derivative $\partial_u W(x,u)$ are continuous. Moreover, there exists a constant $\Lambda>0$ such that \[|\partial_u W(x,u)|\leq \Lambda\quad \text{for all }(x,u)\in M\times\R.\]
\end{itemize}
Even when $G$ and $W$ are smooth, a stationary solution $u_-$ is in general only Lipschitz continuous. Consequently, the frozen Hamiltonian $H_-(x,p):=G(x,p)+W(x,u_-(x))$ need not be smooth. The uniform boundedness of $\partial_uW$ ensures that $H$ is uniformly Lipschitz continuous in $u$, which is used in comparison and stability estimates for the solution semigroup. The continuity of $\partial_uW$ gives a uniform first-order expansion of $W$ near $u_-$, which plays an important role in our stability and instability criteria.

To investigate the Lyapunov stability and instability of fixed points of $T^-_t$ under PDE-type assumptions, the main difficulty arises from the low regularity of the Hamiltonian, which prevents the construction of a contact Hamiltonian flow and, in particular, the derivation of estimates on the velocities of minimizing curves. To address this difficulty, we restrict ourselves to Hamiltonians of the form \eqref{H=}. The treatment of general Hamiltonians is left for future work.  The proofs rely on techniques developed in \cite{NWY,NW}.

 For Hamiltonians independent of $u$, large-time convergence has been studied under suitable assumptions; see \cite{F2,NR}. This convergence result has been further extended to second-order equations and to weakly coupled systems of Hamilton--Jacobi equations \cite{CGM,CLL}. Extending the results of the present paper to these more general settings will be addressed in future work. When $H(x,p,u)$ is nondecreasing in $u$, the solution $u(x,t)$ still converges uniformly to a solution of \eqref{se} as $t\to+\infty$ under additional structural assumptions; see \cite{S}. In contrast, when $H(x,p,u)$ is decreasing in $u$, such convergence generally fails, and the large-time behavior becomes more complicated; see \cite{WWY,WYZ1,WYZ2}. We also refer the reader to \cite{NW,XYZ} for further results on the large-time behavior of \eqref{ee} in the absence of monotonicity.

In \cite{CF,Da2}, the authors study the vanishing discount problem under assumptions involving the integral of $\partial_u H$ with respect to all Mather measures. These conditions are closely related to \eqref{A} derived in the present paper. Their results show that such assumptions play an important role in the asymptotic analysis of viscosity solutions for the vanishing discount problem. Moreover, in \cite[Theorems 1.7 and 1.8]{CF}, such conditions are crucial for the uniqueness of solutions, and in \cite[Theorem 3.1]{Da2}, they are shown to be important for the existence of solutions. The results of the present paper provide a PDE perspective on these assumptions. We also note that in \cite{NYZ}, the authors use large-time behavior to establish the convergence of viscosity solutions in the vanishing discount problem, further highlighting the connection with the main topic of this paper. In addition, we express the averaged sign conditions in terms of derivatives of the critical value. Through the characterization of the critical value by minimizing measures, these criteria relate Lyapunov stability to the variation of the minimal average action in Aubry--Mather theory \cite{M}. Finally, in Proposition \ref{homo}, we also employ large-time behavior to obtain a convergence rate for a homogenization problem. These results illustrate connections between Aubry--Mather theory and various aspects of asymptotic analysis of viscosity solutions, including large-time behavior, the vanishing discount problem, and homogenization.

\medskip

We define the Lagrangian $L:TM\to\R\cup\{+\infty\}$ associated with $G$ by
\[L(x,v):=\sup_{p\in T^*_xM}\big(\langle p,v\rangle _x-G(x,p)\big),\]
where $\langle \cdot,\cdot\rangle _x$ denotes the dual pairing between $T_xM$ and $T^*_xM$.  It is well known that $L(x,v)$ is lower semicontinuous on $TM$, convex in $v$, and uniformly superlinear in $v$ in the extended-real-valued sense. In particular, $L$ may take the value $+\infty$.

Recall that there exists a unique constant $c$ such that
\begin{equation}\label{G=c}
G(x,Du(x))=c\quad \text{for }x\in M
\end{equation}
admits solutions. This constant is called the critical value of $G$ and is denoted by $c(G)$. It can also be characterized by
\begin{equation}\label{cG}
c(G):=\inf\big\{c\in\R:\ G(x,Du)=c \text{ admits a viscosity subsolution}\big\}.
\end{equation}

\medskip

Let $u_-\in C(M)$ be a solution of \eqref{se}. To study Lyapunov stability and instability of the solution $u_-$, we define
\[c(\ep):=c\big(H(x,p,u_-(x)+\ep)\big),\]
that is, the critical value associated with the Hamiltonian $H(x,p,u_-(x)+\ep)$. To obtain Lyapunov stability, we impose the following assumptions. These conditions can be viewed as Lyapunov-type differential inequalities along the associated Hamilton--Jacobi flow, and will later be complemented by analogous conditions for instability and uniqueness.
\begin{prop}\label{A3}
Assume {\rm (A1)--(A2)}. Then $c(\ep)$ is Lipschitz continuous, and the left derivative $D^-c(\ep)|_{\ep=0}$ exists. Consider the following conditions:
\begin{itemize}
\item[(A3)$_1$] There exists a constant $\zeta>0$ such that
\[c\big(H(x,p,u_-(x))-\zeta \cdot \partial_u W(x,u_-(x))\big)<0.\]

\item[(A3)$_2$] There exist constants $\eta,\zeta>0$ and a Lipschitz continuous function $w(x)$ such that
\[H(x,Dw,u_-(x))-\zeta \cdot \partial_u W(x,u_-(x))\leq -\eta \quad \text{for a.e. } x\in M.\]

\item[(A3)$_3$] The left derivative satisfies $D^-c(\ep)|_{\ep=0}>0$.
\end{itemize}
Then {\rm (A3)$_1$--(A3)$_3$} are equivalent.
\end{prop}
\begin{thm}\label{thm1}
Let $u_-\in C(M)$ be a solution of \eqref{se}. Assume {\rm (A1)--(A2)} and one of {\rm (A3)$_1$--(A3)$_3$}. Then $u_-$ is locally asymptotically stable. More precisely, there exists $\Delta>0$ such that for any $\varphi\in C(M,\mathbb{R})$ with $\|\varphi-u_-\|_\infty\le \Delta$, we have
\[\limsup_{t\to\infty}\frac{\ln \|u(\cdot,t)-u_-(\cdot)\|_\infty}{t}\leq -A\]
for some constant $A>0$ depending on $H$ and $u_-$, where $u(x,t)$ denotes the solution of \eqref{ee}.
\end{thm}
We now present an example where the equation is not necessarily monotone in the unknown function. This example provides an explicit situation in which (A3)$_2$ can be verified directly, including Hamiltonians outside the $C^3$ setting of \cite{RWY}.
\begin{ex}\label{ex}
Let $\zeta,\theta>0$ and let $\phi\in C^1(M)$. Consider
\[-(|D\phi|^2-\theta)u+\zeta|Du|^2+(|D\phi|^2-\theta)\phi-\zeta|D\phi|^2=0\quad \text{for }x\in M.\]
Then $u_-=\phi$ is a solution of the above equation and is locally asymptotically stable.  In particular, on $M=\mathbb T$, one may choose a $C^1$ function $\phi$ that is not $C^2$, such as $\phi(x)=\sin(2\pi x)|\sin(2\pi x)|$, and $0<\theta<\max_{x\in M}|D\phi(x)|^2$. The coefficient of $u$ then changes sign, while the Hamiltonian need not satisfy the $C^3$ regularity hypothesis of \cite{RWY}.

\begin{samepage}
 A related example with an eikonal term is obtained on $M=\T$ by taking
$\phi(x)=\sin(2\pi x)$ and $0<\theta<2\pi$ in
\[
u_t+|Du|-|D\phi(x)|+(\theta-|D\phi(x)|)(u-\phi(x))=0.
\]
Here $u_-=\phi$ is stationary and $\partial_uW=\theta-|D\phi|$ changes sign. For $0<\zeta<1$, the choice $w=0$ gives
\[
|Dw|-(1-\zeta)|D\phi|-\zeta\theta\leq -\zeta\theta,
\]
so {\rm (A3)$_2$} holds. The Hamiltonian is convex and coercive in $p$, but is not differentiable at $p=0$, and thus is not covered by \cite{RWY}.
\end{samepage}
\end{ex}

We next give a criterion for Lyapunov instability.

\begin{thm}\label{thm2}
Let $u_-\in C(M)$ be a solution of \eqref{se}. Assume {\rm (A1)--(A2)} and
\[D^-c(\ep)|_{\ep=0}<0.\]
Then $u_-$ is Lyapunov unstable. More precisely, there exists $\Delta>0$ such that
\[\limsup_{t\to+\infty}\|T^-_t(u_--\delta)-u_-\|_\infty\geq\Delta
\quad\text{for every }\delta\in(0,\Delta).\]
By Lemma \ref{D-c}, the assumption is equivalent to
\[\int_{TM}\partial_uW(x,u_-(x))\,d\tilde\mu(x,v)<0
\quad\text{for some }\tilde\mu\in\widetilde{\mathfrak M}_-.\]
\end{thm}
In Theorem \ref{thm2}, negativity for a single minimizing measure suffices for instability. We use the superposition principle of \cite{Bernard} to construct calibrated curves from a minimizing closed measure. By Jensen's inequality and the convexity of $G$, replacing the velocity distribution by its barycenter does not increase the action. Since $\partial_uW(x,u_-(x))$ depends only on $x$, its integral is unchanged under this replacement.

Finally, to obtain the uniqueness result and the global Lyapunov stability, we impose the following assumptions.
\begin{prop}\label{A5}
Assume {\rm (A1)--(A2)}. Consider the following conditions:
\begin{itemize}
\item[(A5)$_1$] For any subsolution $w(x)$ of \eqref{se}, there exists a constant $\zeta>0$ such that
\[c\big(H(x,p,w(x))-\zeta \cdot \partial_u W(x,w(x))\big)<0.\]

\item[(A5)$_2$] For any subsolution $w(x)$ of \eqref{se}, there exist constants $\eta,\zeta>0$ and a Lipschitz continuous function $\tilde w(x)$ such that
\[H(x,D\tilde w,w(x))-\zeta \cdot \partial_u W(x,w(x))\leq -\eta \quad \text{for a.e. } x\in M.\]
\end{itemize}
Then {\rm (A5)$_1$} and {\rm (A5)$_2$} are equivalent.
\end{prop}
\begin{thm}\label{thm3}
Assume {\rm (A1)--(A2)} and one of {\rm (A5)$_1$--(A5)$_2$}. Then \eqref{se} admits at most one solution. Moreover, if such a solution $u_-$ exists, then it is globally asymptotically stable, that is, for all $\varphi\in C(M)$, we have $\lim_{t\to+\infty}T^-_t\varphi=u_-$ uniformly.
\end{thm}

As an application of Theorem \ref{thm3}, we obtain the following corollary.
\begin{cor}\label{a>0A}
Assume {\rm (A1)} and let $a(x)\in C(M)$. Let $\mathcal A$ be the projected Aubry set associated with $G$, defined in \eqref{defA}. Assume $a(x)\geq 0$. If $a(x)>0$ for all $x\in \mathcal A$, then
\begin{equation}\label{au}
a(x)u+G(x,Du)=c(G)\quad \text{for }x\in M
\end{equation}
admits a unique solution, which is globally asymptotically stable.
\end{cor}

The conclusion of Corollary \ref{a>0A} can also be deduced by combining the convergence results for $a(x)\geq 0$ (see \cite{S}) with the uniqueness result in \cite[Proposition 4.1]{Z}. However, our proof is completely independent and relies solely on solution semigroups.

 We also record a quantitative homogenization result for strictly increasing Hamiltonians. Its proof uses finite-time homogenization and the contraction provided by strict monotonicity. This result can be viewed as a nonlinear generalization of \cite[Theorem 1.2]{HJ} and a multiscale extension of the problem studied in \cite{LTZ}. In \cite{E}, only convergence (without rate) was obtained for this nonlinear case. The curve-cutting approach, originally introduced in \cite{TY}, has played a central role in the quantitative homogenization of Hamilton--Jacobi equations. In contrast to \cite{HJ}, we do not have an explicit representation formula for the solution $u^\ep$ of \eqref{eep}, and therefore the curve-cutting argument cannot be applied directly to derive a convergence rate. The result in \cite[Theorem 1.1]{MN}, which relies on the curve-cutting method, is used in the analysis. Together with an additional ingredient from the large-time behavior, we obtain a convergence rate in the present nonlinear setting. Moreover, unlike \cite{LTZ}, the multiscale nature of \eqref{eep} prevents the use of rescaling arguments and the comparison principle to obtain an $O(\ep)$ rate.
\begin{prop}\label{homo}
Let $H\in \BUC(\R^n\times\T^n\times B(0,R)\times{[-R,R]})$ for each $R>0$. Assume that:
\begin{itemize}
\item[(1)] There exists a constant $\Lip(H)>0$ such that
\[
|H(x_1,y,p,u)-H(x_2,y,p,u)|
\leq \Lip(H)|x_1-x_2|
\]
for all $x_1,x_2\in\mathbb R^n$ and
$(y,p,u)\in\mathbb T^n\times\mathbb R^n\times\mathbb R$.
\item[(2)] $p\mapsto H(x,y,p,u)$ is convex.
\item[(3)] $\lim_{|p|\to+\infty}\inf_{x\in \R^n,y\in \T^n}H(x,y,p,0)=+\infty$.
\item[(4)] there exist constants $\Lambda_1,\Lambda_2>0$ such that for every $u_1<u_2$,
\begin{equation}\label{L1L2}
\Lambda_1(u_2-u_1)
\leq
H(x,y,p,u_2)-H(x,y,p,u_1)
\leq
\Lambda_2(u_2-u_1)
\end{equation}
for all $(x,y,p)\in\R^n\times\T^n\times\R^n$.
\end{itemize}
Let $u^\ep$ be the unique solution in $\BUC(\R^n)$ of
\begin{equation}\label{eep}
H\Big(x,\frac{x}{\ep},Du^\ep,u^\ep\Big)=0\quad \text{for }x\in\R^n.
\end{equation}
Define the effective Hamiltonian as follows: for each $x,p\in\R^n$ and $c\in \R$, there is a unique constant $\ol H(x,p,c)$ such that
\[H(x,y,p+D_y w,c)=\ol H(x,p,c)\quad \text{for }y\in\T^n\]
admits solutions. Then there exists a unique solution $\ol u\in \BUC(\R^n)$ of
\begin{equation}\label{olH}
\ol H(x,D\ol u,\ol u)=0\quad \text{for }x\in\R^n,
\end{equation}
and there exists a constant $C>0$, depending only on $H$, such that
\[\|u^\ep-\ol u\|_\infty\leq C\sqrt{\ep}.\]
\end{prop}

\subsection*{Notation}
Let $X$ and $Y$ be metric spaces or smooth manifolds. We denote by $C(X)$ (resp. $\BUC(X)$, $C_{\rm b}(X)$ and $C^1(X)$) the spaces of continuous (resp. bounded uniformly continuous, bounded continuous and $C^1$) functions on $X$. We denote by $W^{m,p}(X)$ (resp. $W^{m,p}(X,Y)$) the Sobolev space of functions from $X$ to $\R$ (resp. from $X$ to $Y$) with weak derivatives up to order $m$ in $L^p(X)$ (resp. $L^p(X,Y)$), and by $\|\cdot\|_{W^{m,p}}$ the corresponding norm. The set $B(x,r)\subset \R^n$ stands for the open ball centered at $x$ with the radius $r$. For $x,y\in M$, we denote by $d(x,y)$ the Riemannian distance between $x$ and $y$. We write $\diam(M)$ for the diameter of $M$.

\subsection*{Organization}
In Section \ref{Sec2}, we collect several results from weak KAM theory for \eqref{G=c}, as well as techniques of contact-type Hamilton--Jacobi equations \eqref{ee} and \eqref{se}. In Section \ref{Sec3}, we prove Proposition \ref{A3}, Theorem \ref{thm1} and Example \ref{ex}. Section \ref{Sec4} is devoted to the proof of Theorem \ref{thm2}. Finally, in Section \ref{Sec5}, we prove Theorem \ref{thm3}, Corollary \ref{a>0A} and Proposition \ref{homo}.

\section{Preliminaries}\label{Sec2}

We first recall several results from the appendix of the arXiv version of \cite{Da1}. For weak KAM theory associated with \eqref{G=c} under standard PDE assumptions, we also refer the reader to \cite{FS,F}. Throughout this part, we assume (A1) and consider \eqref{G=c}. Since all subsolutions of \eqref{G=c} are equi-Lipschitz continuous, one may modify $G$ outside a sufficiently large ball so that the map $p\mapsto G(x,p)$ becomes uniformly superlinear. We then introduce the minimal action function
\[h_t(x,y):=\inf\bigg\{\int_{-t}^0L(\gamma(s),\dot\gamma(s))\, ds:\ \gamma\in{\rm AC}([-t,0],M),\ \gamma(-t)=x,\ \gamma(0)=y\bigg\},\]
where ${\rm AC}([-t,0],M)$ denotes the set of absolutely continuous curves defined on $[-t,0]$ with values in $M$. The Peierls barrier is defined by
\[h(x,y):=\liminf_{t\to+\infty}(h_t(x,y)+c(G)t).\]
The projected Aubry set is the closed set defined by
\begin{equation}\label{defA}
\mathcal A:=\{y\in M:\ h(y,y)=0\}.
\end{equation}
\begin{lem}\label{lem:G=c}
The projected Aubry set $\mathcal A$ is non-empty. Let $y\in\mathcal A$. Then, for every subsolution $w$ of \eqref{G=c}, one has
\[G(y,p)=c(G)\quad \text{for every }p\in \partial^-w(y),\]
where $\partial^-$ denotes the subdifferential.
\end{lem}
\begin{defn}
For a Lipschitz continuous function $u$, define
\[\mathcal L(u):=\{(x,v)\in TM:\ \partial^+_{\rm C}u(x,v)=L(x,v)+c(G)\},\]
where
\[\partial^+_{\rm C}u(x,v):=\sup\{p(v):\ p\in\partial_{\rm C} u(x)\},\]
and $\partial_{\rm C}$ denotes the Clarke generalized gradient. The Aubry set is then defined by
\[\widetilde {\mathcal A}:=\bigcap_{u}\mathcal L(u),\]
where the intersection is taken over all subsolutions of \eqref{G=c}.
\end{defn}

\begin{defn}
A Borel probability measure $\tilde{\mu}$ on $TM$ is said to be closed if
\begin{itemize}
\item [(1)] $\int_{TM}|v|\, d\tilde{\mu}(x,v)<+\infty$;
\item [(2)] for every function $f\in C^1(M)$, one has $\int_{TM}\langle Df(x),v\rangle_x \, d\tilde{\mu}(x,v)=0$.
\end{itemize}
\end{defn}

\begin{prop}\label{Mather}
The following holds
\begin{equation}\label{muc}
\min_{\tilde{\mu}}\int_{TM}L(x,v)\, d\tilde{\mu}=-c(G),
\end{equation}
where the minimum is taken among all closed measures on $TM$. Measures achieving the minimum are called Mather measures. We denote by $\widetilde{\mathfrak M}$ the set of all Mather measures, which is convex and compact in the weak topology. All Mather measures are supported on $\widetilde{\mathcal A}$. The Mather set is defined as
\[\widetilde{\mathcal M}=\overline{\bigcup_{\tilde{\mu}\in\tilde{\mathfrak M}}\textrm{supp}(\tilde{\mu})}.\]
The projected Mather set is $\mathcal M=\pi(\widetilde{\mathcal M})$.
\end{prop}

The following result is classical in the calculus of variations.
\begin{lem}\cite[Theorem 3.5]{One}.\label{lem:ls}
Let $J$ be a bounded interval of $\mathbb R$. Assume that $L(x,v)$ is lower semicontinuous, convex in $v$, and bounded from below. Then the integral functional
\begin{equation*}
  \mathcal L(\gamma):=\int_J L(\gamma(s),\dot \gamma(s))\, ds
\end{equation*}
is sequentially weakly lower semicontinuous in $W^{1,1}(J,M)$, that is, if a sequence $(\gamma_n)_n$ converges weakly to $\gamma$ in $W^{1,1}(J,M)$, then
\[\mathcal L(\gamma)\leqslant\liminf_{n\to+\infty} \mathcal L(\gamma_n).\]
In particular, the same inequality holds if $(\gamma_n)_n$ converges uniformly to $\gamma$ and is equi-Lipschitz continuous.
\end{lem}

We finally collect several results given in \cite{NWY,NW}. We also refer the reader to \cite{CCJWY,IWWY,WWY1} for related results on contact Hamiltonian systems and contact-type Hamilton--Jacobi equations. Assume (A1)--(A2). The Lagrangian associated with $H$ is defined by
\[L_H(x,v,u):=\sup_{p\in T^*_xM}\big(\langle p,v\rangle_x-H(x,p,u)\big).\]
\begin{prop}\label{ppT-}
The backward semigroup is given by
\begin{equation}\label{T-}
T^-_t\varphi(x)=\inf\bigg\{\varphi(\gamma(0))+\int_0^tL_H(\gamma(s),\dot\gamma(s),T^-_s\varphi(\gamma(s)))\, ds\bigg\},
\end{equation}
where the infimum is taken over all absolutely continuous curves $\gamma:[0,t]\to M$ satisfying $\gamma(t)=x$. The semigroup $T^-_t$ is well-defined for every $\varphi\in C(M)$. If $\varphi$ is continuous, then $u(x,t):=T^-_t\varphi(x)$ is the unique continuous viscosity solution of \eqref{ee}. Moreover, if $\varphi$ is Lipschitz continuous, then $u(x,t)$ is locally Lipschitz continuous on $M\times [0,+\infty)$.
\end{prop}

\begin{prop}\label{Tprop}
The backward semigroup satisfies the following properties:
\begin{itemize}
\item[(1)] For $\varphi_1,\ \varphi_2\in C(M)$, if $\varphi_1(x)<\varphi_2(x)$ for all $x\in M$, then
\[T^-_t\varphi_1(x)<T^-_t\varphi_2(x)\quad \text{for all }(x,t)\in M\times(0,+\infty);\]
\item[(2)] For any $\varphi_1,\varphi_2\in C(M)$, 
\[\|T^-_t\varphi_1-T^-_t\varphi_2\|_\infty\leq e^{\Lambda t}\|\varphi_1-\varphi_2\|_\infty.\]
\end{itemize}
\end{prop}
\begin{defn}\label{defb}
A function $u_-\in C(M)$ is called a backward weak KAM solution of \eqref{se} if the following hold:
\begin{itemize}
\item[(1)] For every absolutely continuous curve $\gamma:[t',t]\to M$, 
\[u_-(\gamma(t))-u_-(\gamma(t'))\leq \int_{t'}^tL_H(\gamma(s),\dot\gamma(s),u_-(\gamma(s)))\, ds.\]
In this case, we say that $u_-$ is dominated by $L_H$, and write $u_-\prec L_H$.
\item [(2)] For each $x\in M$, there exists an absolutely continuous curve $\gamma_-:(-\infty, 0] \to M$ with $\gamma_-(0) = x$ such that for all $t>0$,
\[u_-(x)-u_-(\gamma_-(-t))=\int_{-t}^0L_H(\gamma_-(s),\dot\gamma_-(s),u_-(\gamma_-(s)))\, ds.\]
Such curves are called $u_-$-calibrated curves.
\end{itemize}
\end{defn}

\begin{prop}
Let $u_-\in C(M)$. The following statements are equivalent:
\begin{itemize}
\item[(1)] $u_-$ is a fixed point of $T^-_t$;
\item[(2)] $u_-$ is a backward weak KAM solution of \eqref{se};
\item[(3)] $u_-$ is a viscosity solution of \eqref{se}.
\end{itemize}
\end{prop}

To prove Theorem \ref{thm3}, we introduce the forward semigroup
\[T^+_t\varphi(x)=\sup\bigg\{\varphi(\gamma(t))-\int_0^tL_H(\gamma(s),\dot\gamma(s),T^+_{t-s}\varphi(\gamma(s)))\, ds\bigg\},\]
where the supremum is taken over all absolutely continuous curves $\gamma:[0,t]\to M$ with $\gamma(0)=x$. The forward semigroup is well-defined for $\varphi\in C(M)$ and enjoys properties analogous to those in Proposition \ref{Tprop}. In particular, one can define forward weak KAM solutions of \eqref{se} as fixed points of $T^+_t$.
\begin{prop}\label{limsup}
If $T^-_t\varphi$ is uniformly bounded in $t$, then the lower half-relaxed limit
\[\check \varphi(x)=\lim_{r\to 0^+}\inf\Big\{T^-_t\varphi(y):\ d(x,y)<r,\ t>\frac{1}{r}\Big\}\]
is a Lipschitz continuous solution of \eqref{se}; see \cite[Theorem 6.1]{IWWY}. If $T^-_t\varphi$ is uniformly bounded in $t$, then the upper
half-relaxed limit
\[
\bar\varphi(x)
:=
\lim_{r\to0^+}\sup
\Big\{
T^-_t\varphi(y):
d(x,y)<r,\ t>\frac1r
\Big\}
\]
is a Lipschitz continuous subsolution of \eqref{se}.  To verify the subsolution assertion, set $v^\ep(x,s):=T^-_{s/\ep}\varphi(x)$ for $s>0$. Then $v^\ep$ solves
$\ep\partial_s v^\ep+H(x,Dv^\ep,v^\ep)=0$. Its upper half-relaxed limit as $\ep\to0^+$ is $\bar\varphi(x)$, independent of $s$. Stability of viscosity subsolutions therefore shows that $\bar\varphi$ is a subsolution of \eqref{se}. Its Lipschitz continuity follows from its boundedness and the coercivity of $G$.

If $T^+_t\varphi$ is uniformly bounded in $t$, then the upper half-relaxed limit
\[\hat \varphi(x)=\lim_{r\to 0^+}\sup\Big\{T^+_t\varphi(y):\ d(x,y)<r,\ t>\frac{1}{r}\Big\}\]
is a Lipschitz continuous forward weak KAM solution of \eqref{se}.
\end{prop}

\begin{prop}\label{u-u+}
Let $u_-$ be a solution of \eqref{se}. Then $T^+_tu_-\leq u_-$ for all $t>0$. The limit $u_+:= \lim_{t\to+\infty}T^+_t u_-$ exists, and $u_+$ is a forward weak KAM solution of \eqref{se}.

Conversely, let $u_+$ be a forward weak KAM solution of \eqref{se}. Then $T^-_tu_+\geq u_+$ for all $t>0$. The limit $u_-:= \lim_{t\to+\infty}T^-_t u_+$ exists, and $u_-$ is a solution of \eqref{se}.
\end{prop}

\begin{prop}\label{T-T+}\cite[Proposition 2.6]{NW}. 
For each $\varphi\in C(M)$, we have \[T^+_t\circ T^-_t\varphi\leq \varphi\leq T^-_t\circ T^+_t\varphi\quad \text{for all }t>0.\]
\end{prop}

\section{Proof of Theorem \ref{thm1}}\label{Sec3}

We begin with the proof of Proposition \ref{A3}.
\begin{lem}\label{>A}
Assume {\rm (A1)--(A2)} and {\rm (A3)$_1$}. There exists a constant $A>0$ such that for every Mather measure $\tilde \mu_-$ associated with the Hamiltonian \[H_-(x,p):=H(x,p,u_-(x))=G(x,p)+W(x,u_-(x)),\] we have
\begin{equation}\label{A}
\int_{TM}\partial_u W(x,u_-(x))\, d\tilde\mu_->A.
\end{equation}
\end{lem}
\begin{proof}
By (A3)$_1$ and \eqref{muc}, we obtain
\[\min_{\tilde \mu}\int_{TM}\Big(L(x,v)-W(x,u_-(x))+\zeta \cdot \partial_u W(x,u_-(x))\Big)\, d\tilde\mu>0,\]
where $\tilde{\mu}$ is taken over all closed measures on $TM$. Let $\tilde \mu_-$ be a Mather measure associated with the Hamiltonian $H_-$. Since $u_-(x)$ is a solution of
\[G(x,Du)+W(x,u_-(x))=0,\]
we have
\[\int_{TM}\Big(L(x,v)-W(x,u_-(x))\Big)\, d\tilde\mu_-= -c\Big(G(x,p)+W(x,u_-(x))\Big)=0.\]
Therefore,
\begin{align*}
&\zeta \int_{TM}\partial_u W(x,u_-(x))\, d\tilde\mu_-
\\ &=\int_{TM}\Big(L(x,v)-W(x,u_-(x))+\zeta\cdot \partial_u W(x,u_-(x))\Big)\, d\tilde\mu_-
\\ &\geq \min_{\tilde \mu}\int_{TM}\Big(L(x,v)-W(x,u_-(x))+\zeta\cdot \partial_u W(x,u_-(x))\Big)\, d\tilde\mu>0.
\end{align*}
Recall that the set of all Mather measures associated with $H_-$, denoted by $\widetilde{\mathfrak M}_-$, is compact in the weak topology. Since $\partial_uW(\cdot,u_-(\cdot))\in C(M)$, the map $\tilde \mu\mapsto \int_{TM}\partial_u W(x,u_-(x))\, d\tilde \mu$ is continuous with respect to the weak topology. Hence, by compactness, the functional attains its minimum over $\widetilde{\mathfrak M}_-$, which yields a constant $A>0$ such that \eqref{A} holds.
\end{proof}

\begin{lem}\label{A31A32}
Assumption {\rm (A3)$_1$} is equivalent to Assumption {\rm (A3)$_2$}.
\end{lem}
\begin{proof}
We first assume (A3)$_1$. There exists a viscosity solution $w(x)$ of
\begin{equation}\label{HWc}
H(x,Dw,u_-(x))-\zeta\cdot \partial_u W(x,u_-(x))=c\Big(H(x,p,u_-(x))-\zeta\cdot \partial_u W(x,u_-(x))\Big),
\end{equation}
where the right-hand side is negative by (A3)$_1$. The coercivity of $G(x,p)$ in $p$ implies that $w(x)$ is Lipschitz continuous. Moreover, since $G(x,p)$ is convex in $p$, $w(x)$ is an a.e.-subsolution of \eqref{HWc}, which implies (A3)$_2$.

\medskip

Now assume (A3)$_2$. Since $G(x,p)$ is convex in $p$, $w(x)$ is a viscosity subsolution of
\[H(x,Dw,u_-(x))-\zeta \cdot \partial_u W(x,u_-(x))\leq -\eta \quad \text{for } x\in M.\]
By \eqref{cG},
\[c\Big(H(x,p,u_-(x))-\zeta\cdot \partial_u W(x,u_-(x))\Big)<0.\]
\end{proof}

\begin{lem}
Assume {\rm (A1)--(A2)}. Then $c(\ep)$ is Lipschitz continuous.
\end{lem}
\begin{proof}
Let $\ep_1,\ep_2\in\R$ and $\lambda>0$. We consider the discounted equations
\[\lambda u^i_\lambda+G(x,Du^i_\lambda)+W(x,u_-(x)+\ep_i)=0,\quad i=1,2.\]
Define
\[w^1_\lambda(x):=u^1_\lambda(x)-\frac{\Lambda|\ep_1-\ep_2|}{\lambda},\]
then
\begin{align*}
&\lambda w^1_\lambda+G(x,Dw^1_\lambda)+W(x,u_-(x)+\ep_2)
\\ &=\lambda u^1_\lambda(x)-\Lambda|\ep_1-\ep_2|+G(x,Du^1_\lambda)+W(x,u_-(x)+\ep_2)
\\ &\leq \lambda u^1_\lambda(x)+G(x,Du^1_\lambda)+W(x,u_-(x)+\ep_1)=0.
\end{align*}
Hence, $w^1_\lambda$ is a subsolution of the equation satisfied by $u^2_\lambda$. By the comparison principle, we obtain
\[u^1_\lambda(x)-\frac{\Lambda|\ep_1-\ep_2|}{\lambda}\leq u^2_\lambda.\]
By symmetry, and multiplying by $\lambda$, we conclude that
\[|\lambda u^1_\lambda-\lambda u^2_\lambda|\leq \Lambda|\ep_1-\ep_2|.\]
According to \cite{LPV}, $\lambda u^i_\lambda \to -c(\ep_i)$ for $i=1,2$. Passing to the limit as $\lambda \to 0$, we deduce that $c(\ep)$ is Lipschitz continuous.
\end{proof}

The following result is analogous to \cite[Theorem 1.2]{TZ}, where the derivative of the effective Hamiltonian with respect to the variable $p$ is studied.
\begin{lem}\label{D-c}
Assume {\rm (A1)--(A2)}. Then
\[D^-c(\ep)|_{\ep=0}=\inf_{\tilde \mu_-\in\widetilde{\mathfrak M}_-}\int_{TM}\partial_u W(x,u_-(x))\, d\tilde\mu_-.\]
Similarly,
\[D^+c(\ep)|_{\ep=0}=\sup_{\tilde \mu_-\in\widetilde{\mathfrak M}_-}\int_{TM}\partial_u W(x,u_-(x))\, d\tilde\mu_-.\]
As a consequence, both {\rm (A3)$_1$} and {\rm (A3)$_2$} imply {\rm (A3)$_3$}.
\end{lem}
\begin{proof}
We only show the formula for $D^-c(\ep)|_{\ep=0}$. The proof of the formula for $D^+c(\ep)|_{\ep=0}$ is similar. Let $\tilde\mu_-\in\widetilde{\mathfrak M}_-$. For each $x\in M$, we have
\begin{align*}
W(x,u_-(x)+\ep)=&W(x,u_-(x))+\partial_u W(x,u_-(x))\ep
\\ &+\int_0^1\Big(\partial_uW(x,u_-(x)+\theta \ep)-\partial_uW(x,u_-(x))\Big)\ep\, d\theta.
\end{align*}
Since $\partial_u W$ is uniformly continuous on $M\times[-\|u_-\|_\infty-1,\|u_-\|_\infty+1]$, there exists a modulus of continuity $\omega$ depending only on $\|u_-\|_\infty$ such that, for
$|\ep|<1$ and $\theta\in[0,1]$,
\[
\left|
\partial_uW(x,u_-(x)+\theta\ep)
-\partial_uW(x,u_-(x))
\right|
\leq \omega(|\ep|).
\]
For $\ep\in(-1,0)$, we obtain
\begin{align*}
-c(\ep)&\leq \int_{TM}\Big(L(x,v)-W(x,u_-(x)+\ep)\Big)\, d\tilde\mu_-
\\ &\leq \int_{TM}\Big(L(x,v)-W(x,u_-(x))-\partial_u W(x,u_-(x))\ep+\omega(|\ep|)|\ep|\Big)\, d\tilde\mu_-
\\ &=-c(0)-\ep \int_{TM}\partial_u W(x,u_-(x))\, d\tilde\mu_-+\omega(|\ep|)|\ep|,
\end{align*}
It follows that
\[c(\ep)-c(0)\geq \ep \int_{TM}\partial_u W(x,u_-(x))\, d\tilde\mu_--\omega(|\ep|)|\ep|\quad \text{for }\ep<0,\]
which implies
\begin{equation}\label{limsupc}
\limsup_{\ep\to 0^-}\frac{c(\ep)-c(0)}{\ep}\leq \inf_{\tilde \mu_-\in\widetilde{\mathfrak M}_-}\int_{TM}\partial_u W(x,u_-(x))\, d\tilde\mu_-.
\end{equation}

\medskip

We now prove the reverse inequality. Let $\tilde \mu_\ep$ be a closed measure such that
\[\int_{TM}\Big(L(x,v)-W(x,u_-(x)+\ep)\Big)\, d\tilde\mu_\ep=-c(\ep).\]
 The minimizing property gives
\[
\int_{TM}L(x,v)\,d\tilde\mu_\ep
=\int_{TM}W(x,u_-(x)+\ep)\,d\tilde\mu_\ep-c(\ep),
\]
which is uniformly bounded for $\ep\in(-1,0)$.
By the uniform superlinearity of $L$, for every $A>0$ there exists $R_A>0$ such that
\[
L(x,v)\geq A|v| \quad\text{for } |v|\geq R_A.
\]
Since $L$ is bounded from below and $\tilde\mu_\ep$ is a probability measure, for $R\geq R_A$ we have
\[
A\int_{\{|v|>R\}}|v|\,d\tilde\mu_\ep
\leq\int_{\{|v|>R\}}L(x,v)\,d\tilde\mu_\ep
\leq\int_{TM}L(x,v)\,d\tilde\mu_\ep+C
\leq C,
\]
where $C>0$, possibly different from line to line, is independent of $A$ and $\ep$.
Taking $A=1$ yields a uniform first-moment bound and hence tightness, while allowing $A$ to be arbitrarily large yields uniform integrability of $|v|$. Thus the measures $\tilde\mu_\ep$ are relatively compact in the weak topology. Let $\ep_n\to 0^-$ be such that
\[\lim_{n\to+\infty}\frac{c(\ep_n)-c(0)}{\ep_n}=\liminf_{\ep\to 0^-}\frac{c(\ep)-c(0)}{\ep}.\]
Then, up to a subsequence, $\tilde \mu_{\ep_n}\to \tilde \mu_0$ in the weak topology. We claim that $\tilde\mu_0\in\widetilde{\mathfrak M}_-$.

\medskip

\noindent {\bf $\tilde\mu_0$ is closed}.  For any $\phi\in C^1(M)$, the function $(x,v)\mapsto\langle D\phi(x),v\rangle_x$ is continuous and grows at most linearly in $|v|$. The uniform integrability established above therefore permits passage to the limit in its integral under weak convergence. Hence,
\[\int_{TM} \langle D\phi(x),v\rangle_x\, d\tilde\mu_0=\lim_{n\to+\infty}\int_{TM} \langle D\phi(x),v\rangle_x\, d\tilde\mu_{\ep_n}=0.\]
\medskip

\noindent {\bf $\tilde\mu_0$ is minimizing}. Since
$W(x,u_-(x)+\ep_n)\to W(x,u_-(x))$ uniformly on $M$, we have
\[
\int_{TM}\Big(L(x,v)-W(x,u_-(x))\Big)\,d\tilde\mu_{\ep_n}
=
-c(\ep_n)+o(1)\to -c(0).
\]
By the lower semicontinuity of $L$ and the weak convergence
$\tilde\mu_{\ep_n}\to \tilde\mu_0$,
\[
\int_{TM}\Big(L(x,v)-W(x,u_-(x))\Big)\,d\tilde\mu_0
\leq -c(0).
\]
On the other hand, since $\tilde\mu_0$ is closed, 
\[
\int_{TM}\Big(L(x,v)-W(x,u_-(x))\Big)\,d\tilde\mu_0
\geq -c(0).
\]
Hence equality holds, and $\tilde\mu_0$ is minimizing.

\medskip
Using the expansion of $W$, we get
\begin{align*}
-c(\ep)&=\int_{TM}\Big(L(x,v)-W(x,u_-(x)+\ep)\Big)\, d\tilde\mu_\ep
\\ &\geq \int_{TM}\Big(L(x,v)-W(x,u_-(x))-\partial_u W(x,u_-(x))\ep-\omega(|\ep|)|\ep|\Big)\, d\tilde\mu_\ep
\\ &\geq -c(0)-\ep \int_{TM}\partial_u W(x,u_-(x))\, d \tilde\mu_\ep-\omega(|\ep|)|\ep|.
\end{align*}
Thus,
\begin{equation}\label{liminfc}
\liminf_{\ep\to 0^-}\frac{c(\ep)-c(0)}{\ep}\geq \int_{TM}\partial_u W(x,u_-(x))\, d \tilde\mu_0\geq  \inf_{\tilde \mu_-\in\widetilde{\mathfrak M}_-}\int_{TM}\partial_u W(x,u_-(x))\, d\tilde\mu_-.
\end{equation}
Combining \eqref{limsupc} and \eqref{liminfc}, we obtain the desired result.
\end{proof}

\noindent {\bf Proof of Proposition \ref{A3}.}
Lemmas \ref{A31A32} and \ref{D-c} give the equivalence of
{\rm (A3)$_1$} and {\rm (A3)$_2$}, and their implication of
{\rm (A3)$_3$}. Conversely, assume {\rm (A3)$_3$}. Write
\[
W(x,u_-(x)-\zeta)
=W(x,u_-(x))-\zeta\partial_uW(x,u_-(x))+r_\zeta(x).
\]
By the fundamental theorem of calculus,
\[
r_\zeta(x)=\int_0^\zeta
\big[\partial_uW(x,u_-(x))
-\partial_uW(x,u_-(x)-s)\big]\,ds.
\]
Since $\partial_uW$ is uniformly continuous on the compact set
$\{(x,r):x\in M,\ |r-u_-(x)|\leq1\}$, we have
\[
\frac{\|r_\zeta\|_\infty}{\zeta}
\leq
\sup_{\substack{x\in M\\0\leq s\leq\zeta}}
|\partial_uW(x,u_-(x))
-\partial_uW(x,u_-(x)-s)|
\longrightarrow0.
\]
Moreover, formula \eqref{cG} implies that
\[
|c(F+r)-c(F)|\leq\|r\|_\infty
\]
for every continuous potential $r=r(x)$ and every Hamiltonian
$F$ satisfying {\rm (A1)}. Indeed,
$F-\|r\|_\infty\leq F+r\leq F+\|r\|_\infty$,
and the critical value is monotone and satisfies
$c(F+a)=c(F)+a$ for every constant $a$.
Applying this estimate to the expansion above yields
\[
\left|
c\big(H_-(x,p)-\zeta\partial_uW(x,u_-(x))\big)
-c(-\zeta)
\right|
\leq\|r_\zeta\|_\infty=o(\zeta).
\]
Since $c(0)=0$ and $D^-c(\ep)|_{\ep=0}>0$, it follows that
\[
c\big(H_-(x,p)-\zeta\partial_uW(x,u_-(x))\big)
=-\zeta D^-c(\ep)|_{\ep=0}+o(\zeta)<0
\]
for all sufficiently small $\zeta>0$.
This proves {\rm (A3)$_1$}.\qed

\medskip

We now begin the proof of Theorem \ref{thm1}.

\begin{lem}\label{tc}
For all $c\geq 0$, there exists a time $t_c>0$ such that for any $u_-$-calibrated curve $\gamma:(-\infty,0]\to M$, we have
\[\frac{1}{t_c}\int_{-t_c}^0 \partial_u W(\gamma(s),u_-(\gamma(s)))\, ds>A-c.\]
\end{lem}
\begin{proof}
We argue by a contradiction. Assume that there exists $c_0\geq 0$, a sequence $t_n\to +\infty$ and a sequence $\gamma_n$ of $u_-$-calibrated curves such that
\begin{equation}\label{as1}
\frac{1}{t_n}\int_{-t_n}^0 \partial_u W(\gamma_n(s),u_-(\gamma_n(s)))\, ds\leq A-c_0.
\end{equation}
We define a sequence of probability measures $\tilde\mu_n$ on $TM$ as
\[\int_{TM}f(x,v)\, d\tilde\mu_n:=\frac{1}{t_n}\int_{-t_n}^0f(\gamma_n(s),\dot\gamma_n(s))\, ds\quad \text{for all }f\in C_{\rm b}(TM).\]
Since $G(x,p)$ is coercive in $p$, $u_-$ is Lipschitz continuous. By arguments similar to those in Lemma \ref{d<ep} below, the sequence $\dot\gamma_n$ is uniformly bounded. Hence, the supports of $\tilde\mu_n$ are contained in a common compact subset of $TM$. By Prokhorov's Theorem, there exists a subsequence (still denoted $\tilde\mu_n$) converging in the weak topology to some probability measure $\tilde\mu_*$. We now show that $\tilde\mu_*\in \widetilde{\mathfrak M}_-$.

\medskip

\noindent {\bf $\tilde\mu_*$ is closed.} Since $\dot\gamma_n$ is uniformly bounded, we know that ${\rm supp}(\tilde\mu_*)$ is bounded. Hence,
\[\int_{TM}|v|\, d\tilde\mu_*<+\infty.\]
For any $\phi\in C^1(M)$, since $\phi$ is bounded, we get
\begin{align*}
\int_{TM}\langle D\phi(x),v\rangle_x\, d\tilde\mu_n&=\frac{1}{t_n}\int_{-t_n}^0\langle D\phi(\gamma_n(s)), \dot\gamma_n(s)\rangle_{\gamma_n(s)}\, ds
\\ &=\frac{\phi(\gamma_n(0))-\phi(\gamma_n(-t_n))}{t_n}\to 0\quad \text{as }n\to+\infty.
\end{align*}

\medskip

\noindent {\bf $\tilde\mu_*$ is minimizing.} Since $\gamma_n$ are $u_-$-calibrated curves and $u_-$ is bounded, we have
\begin{align*}
&\int_{TM}\Big(L(x,v)-W(x,u_-(x))\Big)\, d\tilde\mu_n
\\ &=\frac{1}{t_n}\int_{-t_n}^0\Big(L(\gamma_n(s),\dot\gamma_n(s))-W(\gamma_n(s),u_-(\gamma_n(s)))\Big)\, ds
\\ &=\frac{u_-(\gamma_n(0))-u_-(\gamma_n(-t_n))}{t_n}\to 0\quad \text{as }n\to +\infty.
\end{align*}

By the lower semicontinuity of $L$ and the weak convergence
$\tilde\mu_n\to \tilde\mu_*$,
\[
\int_{TM}\Big(L(x,v)-W(x,u_-(x))\Big)\,d\tilde\mu_*
\leq 0.
\]
On the other hand, since $\tilde\mu_*$ is closed and the critical
value of $H_-(x,p)=G(x,p)+W(x,u_-(x))$ is zero,
\[
\int_{TM}\Big(L(x,v)-W(x,u_-(x))\Big)\,d\tilde\mu_*
\geq 0.
\]
Therefore $\tilde\mu_*$ is minimizing, and hence
$\tilde\mu_*\in\widetilde{\mathfrak M}_-$.

\medskip

By \eqref{as1},
\[\int_{TM}\partial_u W(x,u_-(x))\, d\tilde\mu_*\leq A-c_0,\]
which contradicts \eqref{A}.
\end{proof}

\begin{lem}\label{bardelta}
For all $c>0$, there exists $\bar\delta>0$ such that for any $u_-$-calibrated curve $\gamma:[a,b]\to M$ with $-\infty\leq a<b\leq +\infty$, there holds
\[\bigg|\int_0^1\partial_u W(\gamma(s),u_-(\gamma(s))+\theta\xi(s))\, d\theta-\partial_uW(\gamma(s),u_-(\gamma(s)))\bigg|\leq c\quad \text{for all }s\in[a,b],\]
for any $\xi\in C([a,b])$ with $\|\xi\|_\infty\leq\bar\delta$.
\end{lem}
\begin{proof}
Let $\bar\delta\in (0,1)$ be a constant to be determined later. For any $u_-$-calibrated curve $\gamma:[a,b]\to M$, there is $\theta_s\in (0,1)$ such that
\begin{align*}
&\bigg|\int_0^1\partial_u W(\gamma(s),u_-(\gamma(s))+\theta\xi(s))\, d\theta-\partial_uW(\gamma(s),u_-(\gamma(s)))\bigg|
\\ &=\big|\partial_u W(\gamma(s),u_-(\gamma(s))+\theta_s\xi(s))-\partial_uW(\gamma(s),u_-(\gamma(s)))\big|.
\end{align*}
Since $\|\xi\|_\infty\leq \bar\delta<1$, we have
\[|u_-(\gamma(s))+\theta_s\xi(s)|\leq\|u_-\|_\infty+\|\xi\|_\infty\leq \|u_-\|_\infty+1.\]
Since $\partial_uW$ is continuous, there is a modulus of continuity $\omega$ depending on $u_-$ such that
\[\big|\partial_u W(x,u_1)-\partial_uW(x,u_2)\big|\leq \omega(|u_1-u_2|)\]
for all $x\in M$ and $u_1,u_2\in\R$ with $|u_1|,|u_2|\leq \|u_-\|_\infty+1$. Hence,
\[\big|\partial_u W(\gamma(s),u_-(\gamma(s))+\theta_s\xi(s))-\partial_uW(\gamma(s),u_-(\gamma(s)))\big|\leq \omega(|\theta_s\xi(s)|)\leq \omega(\bar\delta).\]
Finally, we choose $\bar\delta\in (0,1)$ sufficiently small such that $\omega(\bar\delta)\leq c$. This yields the desired estimate.
\end{proof}

Define
\[u_\delta:=u_--\delta\quad \text{for } \delta\geq 0.\]
The following result is the key step in the proof of Theorem \ref{thm1}. In contrast to \cite[Lemma 3.3]{RWY}, we do not have the contact Hamiltonian flow in the present setting. Consequently, no information on the difference between velocities $\dot\gamma$ and $\dot\gamma_\delta$ is available. To overcome this difficulty, we restrict ourselves to the case $H(x,p,u)=G(x,p)+W(x,u)$. In this setting, it suffices to estimate the distance between $\gamma$ and $\gamma_\delta$. The proof relies on the lower semicontinuity property established in Lemma \ref{lem:ls}.
\begin{lem}\label{d<ep}
Fix $t_0>0$. For any $\ep>0$, there exists $\delta_0>0$ such that for any $x\in M$ and any $\delta\in[0,\delta_0]$, there exists a minimizer $\gamma_\delta:[-t_0,0]\to M$ of $T^-_{t_0}u_\delta(x)$ with $\gamma_\delta(0)=x$  (after translating the time interval) and a $u_-$-calibrated curve $\gamma:[-t_0,0]\to M$ satisfying
\[d\Big(\gamma_\delta(s),\gamma(s)\Big)<\ep\quad \text{for all }s\in[-t_0,0].\]
\end{lem}
\begin{proof}
We argue by a contradiction. Suppose that there exist a constant $\ep_0>0$, a sequence $\{\delta_n\}\subset (0,1)$ with $\delta_n\to 0$, and a sequence $\{x_n\}\subset M$ such that for every minimizer $\gamma_n:[-t_0,0]\to M$ of $T^-_{t_0}u_{\delta_n}(x_n)$ with $\gamma_n(0)=x_n$ and every $u_-$-calibrated curve $\gamma:[-t_0,0]\to M$, there exists $s_n\in[-t_0,0]$ such that
\[d\Big( \gamma_n(s_n),\gamma(s_n)\Big)\geq \ep_0.\]

We first show that the family $\{\gamma_n\}$ is equi-Lipschitz continuous. Since $\delta_n\in (0,1)$ and $u_{\delta_n}=u_--\delta_n$, we have
\[\|u_{\delta_n}\|_\infty\leq \|u_-\|_\infty+1,\quad Du_{\delta_n}=Du_-.\]
Arguing as in the proof of \cite[Lemma 2.3]{MN}, there is a constant $\kappa>0$, depending on $\|u_-\|_{W^{1,\infty}}$ and $t_0$, such that
\[\|T^-_t u_\delta(x)\|_{W^{1,\infty}}\leq\kappa\quad \text{on }M\times[0,t_0].\]
Let $t_1>t_2$ with $t_1,t_2\in(0,t_0)$. Using the superlinearity of $L(x,v)$, we obtain
\begin{align*}
&\kappa d(\gamma_n(-t_1),\gamma_n(-t_2))+\kappa(t_1-t_2)\geq T^-_{t_0-t_2}u_{\delta_n}(\gamma_n(-t_2))-T^-_{t_0-t_1}u_{\delta_n}(\gamma_n(-t_1))
\\ &=\int_{-t_1}^{-t_2}\Big(L(\gamma_n(s),\dot\gamma_n(s))-W(\gamma_n(s),T^-_{t_0+s}u_{\delta_n}(\gamma_n(s)))\Big)\, ds
\\ &\geq \int_{-t_1}^{-t_2}\Big((\kappa+1)|\dot\gamma_n(s)|+C_{\kappa+1}-\max_{x\in M,\ |u|\leq \kappa}W(x,u)\Big)\, ds
\\ &\geq (\kappa+1)d(\gamma_n(-t_1),\gamma_n(-t_2))+\Big(C_{\kappa+1}-\max_{x\in M,\ |u|\leq \kappa}W(x,u)\Big)(t_1-t_2),
\end{align*}
for some constant $C_{\kappa+1}\in\R$. It follows that
\[d(\gamma_n(-t_1),\gamma_n(-t_2))\leq \Big(\kappa+\max_{x\in M,\ |u|\leq \kappa}W(x,u)-C_{\kappa+1}\Big)(t_1-t_2),\]
which implies that $\{\dot\gamma_n\}$ is bounded. By the Ascoli-Arzelà theorem, there is a subsequence of $\{\gamma_n\}$ (still denoted $\{\gamma_n\}$) uniformly converging to a Lipschitz continuous curve $\gamma_*:[-t_0,0]\to M$. By Proposition \ref{Tprop} (2), we also have
\[T^-_t u_\delta \to u_-\quad \text{uniformly on }M\times[0,t_0]\quad \text{as }\delta\to 0.\]
By Lemma \ref{lem:ls}, for $t_1>t_2$ with $t_1,t_2\in(0,t_0)$, we obtain
\begin{align*}
&u_-(\gamma_*(-t_2))-u_-(\gamma_*(-t_1))=\lim_{n\to+\infty}\Big(T^-_{t_0-t_2}u_{\delta_n}(\gamma_n(-t_2))-T^-_{t_0-t_1}u_{\delta_n}(\gamma_n(-t_1))\Big)
\\ &=\lim_{n\to+\infty}\int_{-t_1}^{-t_2}\Big(L(\gamma_n(s),\dot\gamma_n(s))-W(\gamma_n(s),T^-_{t_0+s}u_{\delta_n}(\gamma_n(s)))\Big)\, ds
\\ &\geq \int_{-t_1}^{-t_2}\Big(L(\gamma_*(s),\dot\gamma_*(s))-W(\gamma_*(s),u_-(\gamma_*(s)))\Big)\, ds,
\end{align*}
 The reverse inequality follows from the domination property of $u_-$ with respect to $L_H$, so equality holds. Thus $\gamma_*$ is a $u_-$-calibrated curve. The uniform convergence $\gamma_n\to\gamma_*$ contradicts the assumption, completing the proof.
\end{proof}

\begin{cor}\label{deltac}
Fix $t_0>0$. For any $c>0$, there exists a constant $\delta_c>0$ such that for any $x\in M$ and any $\delta\in [0,\delta_c]$, there exist a minimizer $\gamma_\delta:[-t_0,0]\to M$ of $T^-_{t_0}u_\delta(x)$ with $\gamma_\delta(0)=x$, and a $u_-$-calibrated curve $\gamma:[-t_0,0]\to M$, such that for every $s\in[-t_0,0]$,
\[\bigg|\int_0^1\partial_u W(\gamma_\delta(s),u_-(\gamma_\delta(s))-\theta w_\delta(s+t_0))\, d\theta-\partial_uW(\gamma(s),u_-(\gamma(s)))\bigg|\leq c,\]
where
\[w_\delta(\tau):=u_-(\gamma_\delta(\tau-t_0))-T^-_\tau u_\delta(\gamma_\delta(\tau-t_0))\quad \text{for }\tau\in[0,t_0].\]
\end{cor}
\begin{proof}
Let $\delta_c\in (0,e^{-\Lambda t_0})$ be a constant to be determined later. By Proposition \ref{Tprop} (2), 
\[|w_\delta(\tau)|\leq \|u_--T^-_\tau u_\delta\|_\infty\leq e^{\Lambda \tau}\delta\leq e^{\Lambda t_0}\delta_c\leq 1\quad \text{for }\tau\in[0,t_0].\]
Hence,
\[|u_-(\gamma_\delta(s))-\theta w_\delta(s+t_0)|\leq\|u_-\|_\infty+|w_\delta(s+t_0)|\leq \|u_-\|_\infty+1\quad \text{for all }s\in[-t_0,0].\]
Therefore, by the continuity of $\partial_u W$ and the above uniform bound, there is a modulus of continuity $\omega$ depending on $u_-$ and $\theta_s\in (0,1)$ such that
\begin{align*}
&\bigg|\int_0^1\partial_u W(\gamma_\delta(s),u_-(\gamma_\delta(s))-\theta w_\delta(s+t_0))\, d\theta-\partial_uW(\gamma(s),u_-(\gamma(s)))\bigg|
\\ &=\big|\partial_u W(\gamma_\delta(s),u_-(\gamma_\delta(s))-\theta_s w_\delta(s+t_0))-\partial_uW(\gamma(s),u_-(\gamma(s)))\big|
\\ &\leq \omega\Big((\|Du_-\|_\infty+1)d(\gamma_\delta(s),\gamma(s))+e^{\Lambda t_0}\delta_c\Big).
\end{align*}
Finally, choosing $\delta_c>0$ sufficiently small, we conclude from Lemma \ref{d<ep} that
\[(\|Du_-\|_\infty+1)d(\gamma_\delta(s),\gamma(s))+e^{\Lambda t_0}\delta_c\]
can be made sufficiently small so that
\[\omega\Big((\|Du_-\|_\infty+1)d(\gamma_\delta(s),\gamma(s))+e^{\Lambda t_0}\delta_c\Big)\leq c.\]
This yields the desired estimate.
\end{proof}

\noindent {\bf Proof of Theorem \ref{thm1}.} By Proposition \ref{A3}, it suffices to assume {\rm (A3)$_1$}. Let $A>0$ be the constant given in Lemma \ref{>A}. We take $c\in (0,\frac{A}{2})$ and let $t_c>0$ be the time given in Lemma \ref{tc}.

\medskip

\noindent {\bf Step 1.} Define
\[\delta^c:=e^{-\Lambda t_c}\bar\delta,\quad u^\delta:=u_-+\delta,\]
where $\delta\in[0,\delta^c]$ and $\bar\delta>0$ is the constant given in Lemma \ref{bardelta}. We aim to estimate $\|T^-_{t_c}u^\delta-u_-\|_\infty$. For any $x\in M$, let $\gamma:[-t_c,0]\to M$ be a $u_-$-calibrated curve with $\gamma(0)=x$. Define
\[v(s):=u_-(\gamma(s-t_c)),\quad v^\delta(s):=T^-_s u^\delta(\gamma(s-t_c)),\quad s\in [0,t_c].\]
For any $s,\Delta s$ such that $s+\Delta s\in [0,t_c]$, we have
\[v(s+\Delta s)=v(s)+\int_s^{s+\Delta s}\Big(L(\gamma(\tau-t_c),\dot\gamma(\tau-t_c))-W(\gamma(\tau-t_c),v(\tau))\Big)\, d\tau,\]
and
\[v^\delta(s+\Delta s)\leq v^\delta(s)+\int_s^{s+\Delta s}\Big(L(\gamma(\tau-t_c),\dot\gamma(\tau-t_c))-W(\gamma(\tau-t_c),v^\delta(\tau))\Big)\, d\tau.\]
Set
\[\bar v(s):=v^\delta(s)-v(s)\quad \text{for }s\in [0,t_c].\]
By Proposition \ref{Tprop} (1), $\bar v(s)\geq 0$ for all $s\in [0,t_c]$ and for all $\delta\geq 0$. Hence,
\begin{align*}
\bar v(s+\Delta s)&\leq \bar v(s)+\int_s^{s+\Delta s}\Big(W(\gamma(\tau-t_c),v(\tau))-W(\gamma(\tau-t_c),v^\delta(\tau))\Big)\, d\tau
\\ &=\bar v(s)-\int_s^{s+\Delta s}\bar v(\tau)\int_0^1\partial_u W(\gamma(\tau-t_c),v(\tau)+\theta\bar v(\tau))\, d\theta\, d\tau.
\end{align*}
Note that both $u_-$ and $\gamma$ are Lipschitz continuous. Since $Du^\delta=Du_-$, $T^-_tu^\delta(x)$ is also Lipschitz continuous on $M\times[0,t_c]$ by Proposition \ref{ppT-}. Thus, $\bar v(s)$ is Lipschitz continuous and for a.e. $s\in [0,t_c]$, 
\[\dot{\bar v}(s)\leq -\bar v(s)\int_0^1\partial_u W(\gamma(s-t_c),u_-(\gamma(s-t_c))+\theta\bar v(s))\, d\theta.\]
By Proposition \ref{Tprop} (2),
\[|\bar v(s)|\leq \|T^-_s u^\delta-u_-\|_\infty\leq e^{\Lambda s}\delta\leq e^{\Lambda t_c}\delta^c=\bar \delta.\]
By Lemma \ref{bardelta} and $\bar v(s)\geq 0$, we have
\[\dot{\bar v}(s)\leq \bar v(s)\Big(-\partial_u W(\gamma(s-t_c),u_-(\gamma(s-t_c)))+c\Big).\]
Consider the comparison ODE
\begin{equation*}
  \left\{
   \begin{aligned}
   &\dot v_0(s)=v_0(s)\Big(-\partial_u W(\gamma(s-t_c),u_-(\gamma(s-t_c)))+c\Big),
   \\
   &v_0(0)=\delta.
   \end{aligned}
   \right.
\end{equation*}
The solution is
\[v_0(t)=\delta e^{\int_0^t(-\partial_u W(\gamma(s-t_c),u_-(\gamma(s-t_c)))+c)\, ds}.\]
By the comparison principle, we obtain
\[|T^-_{t_c}u^\delta(x)-u_-(x)|=\bar v(t_c)\leq \delta e^{\int_0^{t_c}(-\partial_u W(\gamma(s-t_c),u_-(\gamma(s-t_c)))+c)\, ds}<\delta e^{(-A+2c)t_c}.\]

\medskip

\noindent {\bf Step 2.} Define
\[u_\delta:=u_--\delta,\quad \delta\in[0,\delta_c],\]
where $\delta_c>0$ is the constant given by Corollary \ref{deltac}. We estimate $\|T^-_{t_c}u_\delta-u_-\|_\infty$. For any $x\in M$, let $\gamma_\delta:[-t_c,0]\to M$ be a minimizer of $T^-_{t_c}u_\delta(x)$ with $\gamma_\delta(0)=x$. Define
\[w(s):=u_-(\gamma_\delta(s-t_c)),\quad w_\delta(s):=T^-_s u_\delta(\gamma_\delta(s-t_c)),\quad s\in [0,t_c].\]
For any $s,\Delta s$ such that $s+\Delta s\in [0,t_c]$, we have
\[w(s+\Delta s)\leq w(s)+\int_s^{s+\Delta s}\Big(L(\gamma_\delta(\tau-t_c),\dot\gamma_\delta(\tau-t_c))-W(\gamma_\delta(\tau-t_c),w(\tau))\Big)\, d\tau,\]
and
\[w_\delta(s+\Delta s)=w_\delta(s)+\int_s^{s+\Delta s}\Big(L(\gamma_\delta(\tau-t_c),\dot\gamma_\delta(\tau-t_c))-W(\gamma_\delta(\tau-t_c),w_\delta(\tau))\Big)\, d\tau.\]
Define
\[\bar w(s):=w(s)-w_\delta(s),\quad s\in [0,t_c].\]
By Proposition \ref{Tprop} (1), $\bar w(s)\geq 0$ for all $s\in [0,t_c]$ and for all $\delta\geq 0$. We then obtain
\begin{align*}
\bar w(s+\Delta s)&\leq \bar w(s)+\int_s^{s+\Delta s}\Big(W(\gamma_\delta(\tau-t_c),w_\delta(\tau))-W(\gamma_\delta(\tau-t_c),w(\tau))\Big)\, d\tau
\\ &=\bar w(s)-\int_s^{s+\Delta s}\bar w(\tau)\int_0^1\partial_u W(\gamma_\delta(\tau-t_c),w(\tau)-\theta\bar w(\tau))\, d\theta\, d\tau.
\end{align*}
Since $Du_\delta=Du_-$, $T^-_tu_\delta(x)$ is Lipschitz continuous on $M\times[0,t_c]$. Thus, $\bar w(s)$ is Lipschitz continuous and for a.e. $s\in [0,t_c]$, 
\[\dot{\bar w}(s)\leq -\bar w(s)\int_0^1\partial_u W(\gamma_\delta(s-t_c),u_-(\gamma_\delta(s-t_c))-\theta\bar w(s))\, d\theta.\]
Since $\delta\in[0,\delta_c]$ and $\bar w(s)\geq 0$, by Corollary \ref{deltac}, there exists a $u_-$-calibrated curve $\gamma:[-t_c,0]\to M$ such that
\[\dot{\bar w}(s)\leq \bar w(s)\Big(-\partial_uW(\gamma(s-t_c),u_-(\gamma(s-t_c)))+c\Big)\quad \text{for a.e. }s\in[0,t_c].\]
Using the comparison principle for ODEs, we conclude
\[|T^-_{t_c}u_\delta(x)-u_-(x)|\leq \delta e^{(-A+2c)t_c}.\]

\medskip

\noindent {\bf Step 3.} Let
\[\Delta_c:=\min\{\delta^c,\delta_c\}.\]
We take $\varphi\in C(M)$ satisfying $\|\varphi-u_-\|_\infty=:\delta\leq \Delta_c$. Then
\[u_{\delta_c}\leq u_\delta\leq \varphi\leq u^\delta\leq u^{\delta^c}.\]
By Steps 1 and 2, and Proposition \ref{Tprop} (1), we have
\[\begin{aligned}
u_{\delta e^{(-A+2c)t_c}}&=u_--\delta e^{(-A+2c)t_c}\leq T^-_{t_c}u_\delta\leq T^-_{t_c}\varphi\leq T^-_{t_c}u^\delta\\
&\leq u_-+\delta e^{(-A+2c)t_c}=u^{\delta e^{(-A+2c)t_c}}.
\end{aligned}\]
Since $\delta e^{(-A+2c)t_c}<\delta\leq\Delta_c$, we still have
\[u_--\delta e^{2(-A+2c)t_c}\leq T^-_{t_c}u_{\delta e^{(-A+2c)t_c}}\leq T^-_{2t_c}\varphi\leq T^-_{t_c}u^{\delta e^{(-A+2c)t_c}}\leq u_-+\delta e^{2(-A+2c)t_c}.\]
We iterate $n$ times to get
\[\|T^-_{nt_c}\varphi-u_-\|_\infty\leq \delta e^{n(-A+2c)t_c}\quad \text{for all }n\in\N.\]
For any $t>0$ large, we can decompose
\[t=nt_c+r,\quad n\in\N,\ r\in[0,t_c).\]
It follows that
\begin{equation}\label{t<e}
\begin{aligned}
\|T^-_t\varphi-u_-\|_\infty&\leq e^{\Lambda r}\|T^-_{nt_c}\varphi-u_-\|_\infty
\\ &\leq \delta e^{\Lambda r+n(-A+2c)t_c}=\delta e^{(\Lambda +A-2c)r+(-A+2c)t}
\\ &<\delta e^{(\Lambda +A-2c)t_c+(-A+2c)t}=\|\varphi-u_-\|_\infty C_c e^{(-A+2c)t},
\end{aligned}
\end{equation}
where
\[C_c:=e^{(\Lambda +A-2c)t_c}.\]
Here we note that the constraint $\|\varphi-u_-\|_\infty \leq \Delta_c$ introduced above depends on $c$. It remains to remove this dependence.

Now we take $c=\frac{A}{4}$ and define $\Delta:=\Delta_{\frac{A}{4}}$. Then for any $\varphi\in C(M)$ satisfying $\|\varphi-u_-\|_\infty\leq \Delta$, we have
\begin{equation}\label{A/4}
\|T^-_t\varphi-u_-\|_\infty\leq \|\varphi-u_-\|_\infty C_{\frac{A}{4}} e^{-\frac{A}{2}t},
\end{equation}
for all $t>0$ large. This means that $u_-$ is locally asymptotically stable.

Finally, for any $\varphi\in C(M)$ satisfying $\|\varphi-u_-\|_\infty\leq \Delta$, we take
\[s_c:=\max\Big\{t_{\frac{A}{4}},1,-\frac{2}{A}(\ln \Delta_c-\ln(C_{\frac{A}{4}}\Delta ))\Big\}.\]
Then it is clear that by \eqref{A/4}, $\|T^-_{s_c}\varphi-u_-\|_\infty\leq \Delta_c$. By \eqref{t<e}, 
\[\|T^-_{t+s_c}\varphi-u_-\|_\infty\leq \|T^-_{s_c}\varphi-u_-\|_\infty C_c e^{(-A+2c)t}\leq \Delta_c C_c e^{(-A+2c)t}\quad \text{for all }t>0.\]
Thus,
\[\limsup_{t\to+\infty}\frac{\ln \|T^-_t\varphi-u_-\|_\infty}{t}\leq -A+2c\quad \text{for all }0<c<A/2.\]
By the arbitrariness of $c$, we conclude
\[\limsup_{t\to+\infty}\frac{\ln \|T^-_t\varphi-u_-\|_\infty}{t}\leq -A.\]
\qed

\medskip

\noindent {\bf Proof of Example \ref{ex}.} We take $w=0$. Then
\begin{align*}
&H(x,Dw,u_-(x))-\zeta\cdot \partial_u W(x,u_-(x))
\\ &=-(|D\phi|^2-\theta)\phi+(|D\phi|^2-\theta)\phi-\zeta|D\phi|^2+\zeta (|D\phi|^2-\theta)=-\zeta\theta<0.
\end{align*}
Thus,  (A3)$_2$ holds. By Theorem \ref{thm1}, $u_-$ is locally asymptotically stable.\qed

\section{Proof of Theorem \ref{thm2}}\label{Sec4}

By Lemma \ref{D-c}, there exists $\tilde\mu\in\widetilde{\mathfrak M}_-$ such that,
with $a(x):=\partial_uW(x,u_-(x))$,
\begin{equation}\label{negative-one-measure}
\int_{TM}a(x)\,d\tilde\mu(x,v)=-2\alpha<0.
\end{equation}
We first prove a lemma in the spirit of \cite[Lemma 3.4]{RWY}, using the superposition principle \cite[Theorem 5.8]{Bernard} instead of invariance under a smooth Hamiltonian flow.

\begin{lem}\label{lem:one-calibrated}
For every $T>0$, there exists a $u_-$-calibrated curve $\gamma:[-T,0]\to M$ such that
\[\int_{-T}^0 a(\gamma(s))\,ds\leq-2\alpha T.\]
\end{lem}
\begin{proof}
Let $\nu=\pi_\#\tilde\mu$. Disintegrate the measure $\tilde\mu$ in \eqref{negative-one-measure} as
$\tilde\mu(dx,dv)=\nu(dx)\tilde\mu_x(dv)$, and define
\[b(x)=\int_{T_xM}v\,d\tilde\mu_x(v)\]
for $\nu$-almost every $x$. Extend $b$ to a Borel vector field on $M$. The finite first moment of $\tilde\mu$ ensures that $b\in L^1(\nu)$. Since $\tilde\mu$ is closed,
\[\int_M\langle Df,b\rangle\,d\nu=0\quad\text{for all }f\in C^1(M).\]
Thus $\nu$ solves
$\operatorname{div}(b\nu)=0$. By \cite[Theorem 5.8]{Bernard}, there exists a probability measure $\eta_T$ on absolutely continuous curves $\gamma:[-T,0]\to M$ such that
\[(e_s)_\#\eta_T=\nu\quad\text{for every }s\in[-T,0],\]
where $e_s(\gamma)=\gamma(s)$, and such that, for $\eta_T$-almost every $\gamma$,
\[\dot\gamma(s)=b(\gamma(s))\quad\text{for a.e. }s\in[-T,0].\]

Recall that $L_H$ is defined in Section \ref{Sec2} and satisfies
\[L_H(x,v,u_-(x))=L(x,v)-W(x,u_-(x)).\]
Since $c(H_-)=0$ and $\tilde\mu$ is minimizing,
\[\int_{TM}L_H(x,v,u_-(x))\,d\tilde\mu(x,v)=0.\]
The convexity of $L_H$ in $v$ and Jensen's inequality give
\begin{equation}\label{lem:one-jensen}
\int_M L_H(x,b(x),u_-(x))\,d\nu(x)
\leq\int_{TM}L_H(x,v,u_-(x))\,d\tilde\mu(x,v)=0.
\end{equation}
By domination, for $\eta_T$-almost every $\gamma$,
\[\int_{-T}^0L_H(\gamma(s),\dot\gamma(s),u_-(\gamma(s)))\,ds
-u_-(\gamma(0))+u_-(\gamma(-T))\geq0.\]
Since $\dot\gamma=b(\gamma)$ almost everywhere and $(e_s)_\#\eta_T=\nu$ for all $s\in[-T,0]$, integration yields
\begin{align*}
0&\leq\int\bigg(\int_{-T}^0L_H(\gamma(s),\dot\gamma(s),u_-(\gamma(s)))\,ds
-u_-(\gamma(0))+u_-(\gamma(-T))\bigg)d\eta_T(\gamma)\\
&=T\int_ML_H(x,b(x),u_-(x))\,d\nu(x)\\
&\leq T\int_{TM}L_H(x,v,u_-(x))\,d\tilde\mu(x,v)=0.
\end{align*}
Here the endpoint terms cancel because $(e_{-T})_\#\eta_T=(e_0)_\#\eta_T=\nu$. Since the integrand in the first line is nonnegative, it vanishes for $\eta_T$-almost every $\gamma$. Applying domination on the subintervals of $[-T,0]$ shows that each such $\gamma$ is $u_-$-calibrated on $[-T,0]$.

Similarly, the position marginals give
\begin{align*}
\int\bigg(\int_{-T}^0a(\gamma(s))\,ds\bigg)d\eta_T(\gamma)
&=\int_{-T}^0\int_Ma(x)\,d\nu(x)\,ds\\
&=T\int_{TM}a(x)\,d\tilde\mu(x,v)=-2\alpha T.
\end{align*}
Thus at least one $u_-$-calibrated curve has the required integral bound.
\end{proof}

\noindent {\bf Proof of Theorem \ref{thm2}.}
By the uniform continuity of $\partial_uW$ on a compact neighborhood of the graph of $u_-$, choose $\Delta>0$ so small that
\begin{equation}\label{negative-error}
|\partial_uW(x,r)-a(x)|\leq\alpha
\quad\text{whenever }|r-u_-(x)|\leq\Delta.
\end{equation}
Recall that $u_\delta=u_--\delta$, as defined in Section \ref{Sec3}. Comparison with the supersolution $u_--\delta e^{-\Lambda t}$ yields
\begin{equation}\label{negative-gap}
0<\delta e^{-\Lambda t}\leq u_-(x)-T^-_t u_\delta(x)
\quad\text{for all }x\in M\text{ and }t\geq0.
\end{equation}
Fix $\delta\in(0,\Delta)$. Suppose, contrary to the assertion, that
\[\limsup_{t\to\infty}\|T^-_t u_\delta-u_-\|_\infty<\Delta.\]
Then there exists $t_1\geq0$ such that
\[0\leq u_-(x)-T^-_t u_\delta(x)\leq\Delta
\quad\text{for all }x\in M\text{ and }t\geq t_1.\]
For any $T>0$, choose $\gamma:[-T,0]\to M$ as in Lemma \ref{lem:one-calibrated}, and set
\[\bar w(s)=u_-(\gamma(s))-T^-_{t_1+T+s}u_\delta(\gamma(s)),\quad -T\leq s\leq0.\]
By \eqref{negative-gap},
\[0<\delta e^{-\Lambda(t_1+T+s)}\leq\bar w(s)\leq\Delta\quad\text{for all }s\in[-T,0].\]
In particular, $\bar w(-T)\geq\delta e^{-\Lambda t_1}$. The calibration identity for $u_-$ and the variational inequality in Proposition \ref{ppT-} for $T^-_t u_\delta$ along the same curve give, for $-T\leq s<t\leq0$,
\[\bar w(t)-\bar w(s)\geq\int_s^t
\bigl[W(\gamma(\tau),T^-_{t_1+T+\tau}u_\delta(\gamma(\tau)))
-W(\gamma(\tau),u_-(\gamma(\tau)))\bigr]d\tau.\]
Since $u_-$ and $T^-_t u_\delta$ are locally Lipschitz and $\gamma$ is absolutely continuous, $\bar w$ is absolutely continuous. Hence, for almost every $s\in[-T,0]$,
\[
\dot{\bar w}(s)\geq-\bar w(s)\int_0^1\partial_uW(\gamma(s),u_-(\gamma(s))-\theta\bar w(s))\,d\theta
\geq-(a(\gamma(s))+\alpha)\bar w(s),
\]
where we used \eqref{negative-error}. Gronwall's inequality and Lemma \ref{lem:one-calibrated} yield
\[
\bar w(0)\geq\bar w(-T)\exp\left(-\int_{-T}^0(a(\gamma(s))+\alpha)\,ds\right)
\geq\delta e^{-\Lambda t_1}e^{\alpha T}.
\]
For $T$ sufficiently large this contradicts $\bar w(0)\leq\Delta$, completing the proof.\qed

\section{Proof of Theorem \ref{thm3}}\label{Sec5}

We assume that there exists a solution $u_-$ of \eqref{se}. We aim to prove that $u_-$ is the unique solution of \eqref{se} and is globally asymptotically stable.

Similar to Lemma \ref{A31A32}, one can show that {\rm (A5)$_1$} and {\rm (A5)$_2$} are equivalent. By {\rm (A5)$_1$}, {\rm (A3)$_1$} holds for $u_-$. Then, by Theorem \ref{thm1}, there is $\Delta>0$ such that for any $\delta\in(0,\Delta)$, 
\[
\lim_{t\to+\infty}T^-_tu_\delta=\lim_{t\to+\infty}T^-_t u^\delta=u_-\quad \text{uniformly},
\]
where
\[u_\delta:=u_--\delta,\quad u^\delta:=u_-+\delta.\]
Since $u_->u_\delta$ and $\lim_{t\to+\infty}T^-_t u_\delta=u_-$ uniformly, there exists $\tau>0$ such that
\[u_-\geq T^-_tu_\delta\geq u_\delta \quad \text{for all }t\geq \tau.\]
By Proposition \ref{T-T+},
\begin{equation}\label{thm31}
T^+_t u_\delta\leq T^+_t\circ T^-_tu_\delta\leq u_\delta<u_-<u^\delta \quad \text{for all }t\geq \tau.
\end{equation}

\begin{lem}\label{ubd}
$T^+_tu_\delta$ is unbounded from below.
\end{lem}
\begin{proof}
We argue by contradiction. Assume that $T^+_tu_\delta$ is bounded from below. Then, by \eqref{thm31}, $T^+_tu_\delta$ is in fact bounded. Define the upper half-relaxed limit
\[
v_+(x):=
\lim_{r\to0^+}\sup
\Big\{
T^+_t u_\delta(y):
d(x,y)<r,\ t>\frac1r
\Big\}.
\]
By Proposition \ref{limsup}, $v_+$ is a forward weak KAM solution
of \eqref{se}.
By \eqref{thm31},
\begin{equation}\label{thm32}
v_+\leq u_\delta<u_-.
\end{equation}
Define
\[v_-:=\lim_{t\to+\infty}T^-_tv_+.\]
By Proposition \ref{u-u+}, $v_-$ is a solution of \eqref{se}, and
\begin{equation}\label{thm33}
v_+\leq v_-\leq u_-,
\end{equation}
where we recall that $T^-_tv_+\leq T^-_tu_-=u_-$.

\medskip

\noindent {\bf Case 1.} $v_-=u_-$. Then $\lim_{t\to+\infty}T^-_t v_+=u_-$ uniformly. By \eqref{thm32}, there exists $t_0>0$ such that $v_+<T^-_{t_0}v_+$. Hence, by Proposition \ref{T-T+}, we obtain
\[v_+=T^+_{t_0}v_+<T^+_{t_0}\circ T^-_{t_0}v_+\leq v_+,\]
which leads to a contradiction.

\medskip

\noindent {\bf Case 2.} $v_-\neq u_-$. By Theorem \ref{thm1}, $\|v_--u_-\|_\infty>\Delta$ and since (A3)$_1$ also holds for $v_-$, there is $\Delta_1\in (0,\delta)$ such that
\begin{equation}\label{thm34}
\lim_{t\to+\infty}T^-_t\varphi=v_-\quad \text{for all }\varphi\in C(M)\text{ with }\|\varphi-v_-\|_\infty\leq \Delta_1.
\end{equation}
Let \[\varphi(x):=\max\{v_+(x)+\Delta_1,v_-(x)\}.\]
By \eqref{thm32} and \eqref{thm33}, we have
\[
v_++\Delta_1\leq \varphi\leq v_-+\Delta_1,
\]
and
\begin{equation}\label{thm36}
\varphi\leq \max\{u_\delta+\delta,v_-\}\leq u_-.
\end{equation}
By \eqref{thm34},
\[v_-=\lim_{t\to+\infty}T^-_tv_+\leq \lim_{t\to+\infty}T^-_t\varphi\leq \lim_{t\to+\infty}T^-_t(v_-+\Delta_1)=v_-,\]
which implies
\begin{equation}\label{thm37}
\lim_{t\to+\infty}T^-_t\varphi=v_-.
\end{equation}

We claim that there exists $t_1>0$ such that
\[
T^+_{t_1}u_\delta\leq\varphi.
\]
Otherwise, there exist $t_n\to+\infty$ and $x_n\to x_0$ such that
\[
T^+_{t_n}u_\delta(x_n)>\varphi(x_n).
\]
By the definition of $v_+$ and the continuity of $\varphi$,
\[
v_+(x_0)\geq\varphi(x_0)
\geq v_+(x_0)+\Delta_1,
\]
a contradiction.
By \eqref{thm36}, we then obtain
\[T^-_tu_\delta\leq T^-_t\circ T^-_{t_1}\circ T^+_{t_1}u_\delta\leq T^-_{t+t_1}\varphi\leq T^-_{t+t_1}u_-=u_-.\]
Letting $t\to+\infty$, we get $\lim_{t\to+\infty}T^-_t \varphi=u_-$, which contradicts \eqref{thm37}.
\end{proof}

\begin{lem}\label{Q}
Let $\varphi\in C(M)$. If $T^+_t\varphi$ is unbounded from below, then for any $Q\in\R$, there exists $s>0$ such that $T^+_s\varphi(x)\leq Q$ for all $x\in M$.
\end{lem}
\begin{proof}
By \cite[Lemma 5.2]{Ishii}, there exist $\rho>0$ and $C_L>0$
such that
\[
L(x,v)\leq C_L
\quad\text{for all }x\in M,\ |v|\leq\rho.
\]
Choose $h>0$ sufficiently large so that $\frac{\operatorname{diam}(M)}{h}<\rho$.
From the assumption, there exists a sequence $\{(x_n,t_n)\}$ satisfying $t_n\to +\infty$ and
\[T^+_{t_n}\varphi(x_n)\to -\infty\quad \text{as }n\to +\infty.\]
We argue by contradiction. Assume that there is $Q\in\R$ such that for each $t>0$, there exists $y_t\in M$ such that $T^+_t \varphi(y_t)>Q$. After discarding finitely many terms, we may assume that $t_n>h$
and choose $y_n\in M$ such that
\[T^+_{t_{n}-h}\varphi(y_n)>Q\quad \text{for all }n\in \N.\]
We take a geodesic $\alpha:[0,h]\to M$ satisfying $\alpha(0)=x_n$ and $\alpha(h)=y_n$ with a constant speed. Then $|\dot\alpha|\leq\rho$. Define
\[F(s)=Q-T^+_{t_n-s}\varphi(\alpha(s)),\quad s\in[0,h].\]
Then $F(h)<0$. If $F(0)\leq 0$, then $T^+_{t_n}\varphi(x_n)\geq Q$, which leads to a contradiction. So we only need to consider the case where $F(0)>0$. By continuity, there is $\sigma\in [0,h)$ such that $F(\sigma)=0$ and $F(s)>0$ for all $s\in [0,\sigma)$. Thus, for all $\tau\in [0,\sigma)$,
\begin{align*}
T^+_{t_n-\tau}\varphi(\alpha(\tau))&\geq T^+_{t_{n}-\sigma}\varphi(\alpha(\sigma))-\int_\tau^{\sigma}\Big(L(\alpha(s),\dot\alpha(s))-W(\alpha(s),T^+_{t_n-s}\varphi(\alpha(s)))\Big)\, ds
\\ &\geq Q-\int_\tau^\sigma \Big(C_L+C_{W,Q}+\Lambda F(s)\Big)\, ds,
\end{align*}
where
\[C_{W,Q}:=\max\{|W(x,Q)|:\ x\in M\}.\]
We get
\[F(\tau)\leq  \int_\tau^\sigma\Big(C_L+C_{W,Q}+\Lambda F(s)\Big)\,ds\leq (C_L+C_{W,Q})h+\Lambda\int_\tau^\sigma F(s)\, ds.\]
By Gronwall's inequality, 
\[F(\tau)\leq (C_L+C_{W,Q})he^{\Lambda(\sigma-\tau)}.\]
Letting $\tau=0$, we obtain
\[F(0)=Q-T^+_{t_n}\varphi(x_n)\leq (C_L+C_{W,Q})he^{\Lambda h},\]
which contradicts the assumption that $T^+_{t_n}\varphi(x_n)\to -\infty$.
\end{proof}

A similar argument yields the following result.
\begin{lem}\label{Q2}
Let $\varphi\in C(M)$. If $T^-_t\varphi$ is unbounded from above, then for any $Q\in\R$, there exists $s>0$ such that $T^-_s\varphi(x)\geq Q$ for all $x\in M$.
\end{lem}

\begin{lem}\label{<u-}
For any $\varphi\in C(M)$ with $\varphi\leq u_-$, we have $\lim_{t\to+\infty}T^-_t\varphi=u_-$ uniformly.
\end{lem}
\begin{proof}
Let
\[Q:=\inf_{x\in M}\varphi(x).\]
By Lemmas \ref{ubd} and \ref{Q}, there exists $s>0$ such that
\[T^+_su_\delta\leq Q\leq \varphi\leq u_-.\]
By Proposition \ref{T-T+},
\[T^-_tu_\delta\leq T^-_t\circ T^-_s\circ T^+_su_\delta\leq T^-_{t+s}\varphi\leq T^-_{t+s}u_-=u_-.\]
Letting $t\to+\infty$, we conclude that
\[\lim_{t\to+\infty}T^-_t\varphi=u_-\quad \text{uniformly}.\]
\end{proof}

\begin{lem}\label{uni}
$u_-$ is the unique solution of \eqref{se}.
\end{lem}
\begin{proof}
Assume that there is another solution $w_-$ of \eqref{se}.

\medskip

\noindent {\bf Case 1.} If $w_-\leq u_-$, by Lemma \ref{<u-},
\[\lim_{t\to+\infty}T^-_tw_-=u_-,\]
which implies $w_-=u_-$.

\medskip

\noindent {\bf Case 2.} If $w_-\geq u_-$, since $u_-$ in Lemma \ref{<u-} is an arbitrary solution of \eqref{se}, we obtain
\[\lim_{t\to+\infty}T^-_tu_-=w_-,\]
which implies $u_-=w_-$.

\medskip

\noindent {\bf Case 3.} If there exist $x_1,x_2\in M$ such that
\[w_-(x_1)<u_-(x_1)\quad \text{and} \quad w_-(x_2)>u_-(x_2),\]
we define
\[\ol u(x):=\min\{u_-(x),w_-(x)\},\quad x\in M.\]
By Lemma \ref{<u-},
\begin{equation}\label{thm38}
\lim_{t\to+\infty}T^-_t\ol u=u_-\quad \text{uniformly}.
\end{equation}
Since $w_-(x_1)<u_-(x_1)$, there is $\eta>0$ such that $w_-(x_1)<u_-(x_1)-\eta$. Hence,
\[T^-_t\ol u(x_1)\leq T^-_t w_-(x_1)=w_-(x_1)<u_-(x_1)-\eta,\]
which contradicts \eqref{thm38}.
\end{proof}

We point out that the proof of Lemma \ref{uni} does not rely on semiconcavity, in contrast to \cite{RWY}. This is because the weaker regularity assumptions on $H(x,p,u)$ in our setting do not guarantee such estimates; see, e.g., \cite{CS}.

\begin{lem}\label{>u-}
For any $\varphi\in C(M)$ with $\varphi\geq u_-$, we have $\lim_{t\to+\infty}T^-_t\varphi=u_-$ uniformly.
\end{lem}

\begin{proof}
We first show that $T^-_t\varphi$ is uniformly bounded from above.
Assume by contradiction that it is unbounded from above.

Set
\[\phi:=\varphi+\delta\geq u_-+\delta=u^\delta>u_-.\]
By Proposition \ref{Tprop} (1), $T^-_t\phi$ is also unbounded from
above. Let
\[
N:=\sup_{x\in M}\phi(x)+1.
\]
By Lemma \ref{Q2}, there exists $\sigma>0$ such that
\[
T^-_\sigma\phi(x)\geq N>\phi(x)
\quad\text{for all }x\in M.
\]
By Proposition \ref{T-T+},
\[
T^+_\sigma\phi
\leq T^+_\sigma\circ T^-_\sigma\phi
\leq\phi.
\]
Recalling that $(x,t)\mapsto T^+_t\phi(x)$ is continuous, there is a constant $\ol N>0$ depending on $\sigma$ such that
\[T^+_s\phi(x)<\ol N\quad \text{for all }x\in M\text{ and for all }s\in [0,\sigma].\]
For any $t\geq \sigma$, we write
\[t=n\sigma+s,\quad n\in \N^+,\ s\in [0,\sigma).\]
It follows that
\[
T^+_t\phi=T^+_{(n-1)\sigma+s}\circ T^+_\sigma\phi\leq T^+_{(n-1)\sigma+s}\phi\leq \dots\leq T^+_s\phi\leq \ol N.
\]
Since $\phi\geq u_-$,
\[
T^+_t\phi\geq T^+_tu_-\geq u_+,
\]
and hence $T^+_t\phi$ is uniformly bounded.
Define its upper half-relaxed limit
\[
\phi_+(x):=
\lim_{r\to0^+}\sup
\Big\{
T^+_t\phi(y):
d(x,y)<r,\ t>\frac1r
\Big\}.
\]
By Proposition \ref{limsup}, $\phi_+$ is a forward weak KAM
solution of \eqref{se}.

If
\[
T^+_\tau\phi\leq u^\delta
\]
for some $\tau>0$, then Proposition \ref{T-T+} gives
\[
\phi
\leq T^-_\tau\circ T^+_\tau\phi
\leq T^-_\tau u^\delta.
\]
Therefore,
\[
T^-_t\phi\leq T^-_{t+\tau}u^\delta,
\]
contradicting the unboundedness of $T^-_t\phi$.

Thus, for every $t>0$, there exists $x_t\in M$ such that
\[
T^+_t\phi(x_t)>u^\delta(x_t).
\]
Choose $t_n\to+\infty$ and, up to a subsequence,
$x_n:=x_{t_n}\to x_0$. By the definition of $\phi_+$,
\[
\phi_+(x_0)\geq u^\delta(x_0).
\]
Let
\[
\phi_-:=\lim_{t\to+\infty}T^-_t\phi_+.
\]
By Proposition \ref{u-u+}, $\phi_-$ is a solution of \eqref{se}
and $\phi_-\geq\phi_+$. By Lemma \ref{uni}, $\phi_-=u_-$.
Consequently,
\[
u_-(x_0)=\phi_-(x_0)
\geq\phi_+(x_0)
\geq u^\delta(x_0)>u_-(x_0),
\]
a contradiction.

Hence, $T^-_t\varphi$ is uniformly bounded. Let $\check\varphi$
and $\bar\varphi$ denote, respectively, the lower and upper
half-relaxed limits of $T^-_t\varphi$ defined in Proposition
\ref{limsup}.
By Proposition \ref{limsup}, $\check\varphi$ is a solution of
\eqref{se}. Thus, by Lemma \ref{uni}, $\check\varphi=u_-$.

By the second assertion of Proposition \ref{limsup},
$\bar\varphi$ is a Lipschitz subsolution of \eqref{se}.
Since $T^-_t\varphi\geq u_-$, we have $\bar\varphi\geq u_-$.
Moreover, since $\bar\varphi$ is a subsolution, by \cite[Proposition 2.5]{NW},
\[
T^-_t\bar\varphi\geq\bar\varphi
\quad\text{for all }t>0.
\]
Applying the first part of the proof to
$\bar\varphi\geq u_-$, the family
$T^-_t\bar\varphi$ is uniformly bounded from above.
Therefore its lower half-relaxed limit is a solution of
\eqref{se}, and hence, by Lemma \ref{uni}, it equals $u_-$.
Since $T^-_t\bar\varphi\geq\bar\varphi$,
we obtain $u_-\geq\bar\varphi$.
Thus
\[
\check\varphi
=
\bar\varphi
=
u_-.
\]
Since the lower and upper half-relaxed limits coincide with the
continuous function $u_-$ and $M$ is compact, we conclude that
$T^-_t\varphi\to u_-$
uniformly on $M$.
\end{proof}

\begin{lem}\label{glo}
For any $\varphi\in C(M)$, we have $\lim_{t\to+\infty}T^-_t\varphi=u_-$ uniformly.
\end{lem}
\begin{proof}
It remains to consider the case where $\varphi(x_1)>u_-(x_1)$ and $\varphi(x_2)<u_-(x_2)$ for some points $x_1,x_2\in M$. Let
\[\alpha(x):=\max\{\varphi(x),u_-(x)\},\quad \beta(x):=\min\{\varphi(x),u_-(x)\}.\]
It is direct to see that $\beta\leq u_-\leq \alpha$ and $\beta\leq \varphi\leq \alpha$. Then
\[T^-_t\beta\leq T^-_t\varphi\leq T^-_t\alpha\quad \text{for all }t>0.\]
By Lemmas \ref{<u-} and \ref{>u-}, we conclude
\[\lim_{t\to+\infty}T^-_t\varphi=u_-\quad \text{uniformly}.\]
\end{proof}

\subsection{Proof of Corollary \ref{a>0A}}
\begin{lem}\label{lemv}\cite[Theorem 6.2]{FS}.
There exists a subsolution $v$ of \eqref{G=c} which is both strict and of class $C^\infty$ in $M\backslash \mathcal A$, meaning that
\[G(x,Dv)<c(G) \quad \text{for every }x\in M\backslash \mathcal A.\]
\end{lem}

\medskip

Let $u_0$ be a solution of \eqref{G=c}. Since $a(x)\geq 0$, the functions
\[u^+_0:=u_0+\|u_0\|_\infty\quad \text{and}\quad u^-_0:=u_0-\|u_0\|_\infty\]
are, respectively, a supersolution and a subsolution of \eqref{au}, and satisfy $u^-_0\leq u^+_0$. By Perron's method, there exists a solution $u_-$ of \eqref{au}. Let $U$ be an open neighborhood of $\mathcal A$ such that there exists $a_0>0$ and
\[a(x)\geq a_0\quad \text{for all }x\in U.\]
Let $w$ be any subsolution of \eqref{au}, and set
\[
\zeta:=\|w\|_\infty+1.
\]
Then
\[
w(x)-\zeta\leq -1\qquad\text{for all }x\in M.
\]
We take a subsolution $v$ of \eqref{G=c} as given in Lemma \ref{lemv}. Since $v$ is of class $C^\infty$ in $M\backslash \mathcal A$, there exists a constant $\delta(U)>0$ such that
\begin{equation}\label{delU}
G(x,Dv)\leq c(G)-\delta(U)\quad \text{for all }x\in M\setminus U.
\end{equation}
For a.e. $x\in U$,
\[a(x)w(x)+G(x,Dv)-c(G)-\zeta a(x)\leq a(x)(w(x)-\zeta)\leq -a_0,\]
where for the first inequality, we used $G(x,Dv)\leq c(G)$, and for the second inequality, we used $w(x)-\zeta \leq -1$ and $a(x)\geq a_0$. For $x\in M\setminus U$,
\[a(x)w(x)+G(x,Dv)-c(G)-\zeta a(x)\leq G(x,Dv)-c(G)\leq -\delta(U),\]
where for the first inequality, we used $a(x)\geq 0$ and $w(x)<\zeta$, and for the second inequality, we used \eqref{delU}. Then (A5)$_2$ holds. By Theorem \ref{thm3}, $u_-$ is the unique solution of \eqref{au} and is globally asymptotically stable.

\medskip

The proof is now complete.

\subsection{Proof of Proposition \ref{homo}}

Arguing as in \cite[Lemma 2.2]{E}, one can show
\[\bar H(x,p,u_2)-\bar H(x,p,u_1)
\geq \Lambda_1(u_2-u_1)
\qquad\text{for }u_2\geq u_1.\]
It follows that the constant functions
\[-\frac{\|\ol H(x,0,0)\|_\infty}{\Lambda_1}\quad \text{and}\quad \frac{\|\ol H(x,0,0)\|_\infty}{\Lambda_1}\]
are, respectively, a subsolution and a supersolution of \eqref{olH}. By the comparison principle and Perron's method, there exists a unique solution $\ol u$ in $\BUC(\R^n)$ of \eqref{olH}. Let $v^\ep$ be the unique viscosity solution of
\begin{equation*}
\begin{cases}
v^\ep_t+H(x,\frac{x}{\ep},Dv^\ep,v^\ep)=0\quad \text{for }(x,t)\in \R^n\times (0,\infty),
\\ v^\ep(x,0)=\ol u(x).
\end{cases}
\end{equation*}

 By \eqref{L1L2}, the functions
\[
u^\ep(x)\pm\|\ol u-u^\ep\|_\infty e^{-\Lambda_1t}
\]
are, respectively, a supersolution and a subsolution of the evolutionary equation for $v^\ep$. Since
$u^\ep-\|\ol u-u^\ep\|_\infty\leq\ol u\leq u^\ep+\|\ol u-u^\ep\|_\infty$ at $t=0$, the comparison principle yields
\[
\|v^\ep(\cdot,t)-u^\ep\|_\infty
\leq e^{-\Lambda_1t}\|\ol u-u^\ep\|_\infty.
\]

Now we estimate $|v^\ep(x,t)-\ol u(x)|$. Note that $\ol v(x,t):=\ol u(x)$ solves
\begin{equation*}
\begin{cases}
\ol v_t+\ol H(x,D\ol v,\ol v)=0\quad \text{for }(x,t)\in \R^n\times (0,\infty),
\\ \ol v(x,0)=\ol u(x).
\end{cases}
\end{equation*}
Set
\[
R:=\Lambda_1^{-1}\sup_{x\in\mathbb R^n,\,y\in\mathbb T^n}
|H(x,y,0,0)|.
\]
By \eqref{L1L2}, the constants $-R$ and $R$ are,
respectively, a subsolution and a supersolution of both
\eqref{eep} and \eqref{olH}.
The comparison principle therefore gives
\[
\|\bar u\|_\infty\leq R,\qquad
 \|u^\ep\|_\infty\leq R,\qquad
|v^\ep(x,t)|\leq R
\quad\text{for all }x\in\mathbb R^n,\ t\geq0.
\]
Moreover, the coercivity of $\bar H$ implies that
$\bar u$ is Lipschitz continuous.
Define
\[
\chi_R(r):=\max\{-R,\min\{r,R\}\},
\qquad
H_R(x,y,p,r):=H(x,y,p,\chi_R(r)).
\]
Then $H_R$ satisfies the assumptions of
\cite[Theorem 1.1]{MN}.
Its effective Hamiltonian is
$\bar H_R(x,p,r)=\bar H(x,p,\chi_R(r))$.
Since $v^\ep$ and $\bar u$ take values in $[-R,R]$,
they also solve the corresponding equations with
Hamiltonians $H_R$ and $\bar H_R$, respectively. For each fixed $t>0$, applying \cite[Theorem 1.1]{MN}
to $H_R$, we obtain, for all sufficiently small $\ep$,
\[
|v^\ep(x,t)-\bar u(x)|\leq C(t)\sqrt{\ep}
\quad\text{for all }x\in\mathbb R^n,
\]
where $C(t)>0$ depends only on $t$ and $H$.

Combining the above estimates, we obtain that for all $x\in \R^n$,
\begin{align*}
|u^\ep(x)-\ol u(x)|&\leq |u^\ep(x)-v^\ep(x,t)|+|v^\ep(x,t)-\ol u(x)|
\\ &\leq e^{-\Lambda_1 t}\|\ol u-u^\ep\|_\infty+C(t)\sqrt{\ep}.
\end{align*}
Choosing $t=\frac{\log 2}{\Lambda_1}$, we conclude
\[\|u^\ep-\ol u\|_\infty\leq 2C\Big(\frac{\log 2}{\Lambda_1}\Big)\sqrt{\ep},\]
which completes the proof.

\section*{Acknowledgements}

The authors thank Hiroyoshi Mitake and Hung V. Tran for helpful comments and suggestions, and Liang Jin for sharing an alternative proof of the instability criterion in Theorem \ref{thm2}. They also acknowledge ChatGPT for suggesting Bernard's superposition principle \cite{Bernard}, which is used in the proof of Theorem \ref{thm2}. Panrui Ni is supported by the JSPS grant: KAKENHI \# 26KF0103 and partially supported by the National Natural Science Foundation of China (Grant No. 12571197). Jun Yan is supported by the National Natural Science Foundation of China (Grant Nos. 12171096, 12571197, 12231010).

\section*{Declarations}

\noindent {\bf Conflict of interest statement:} The authors state that there is no conflict of interest.

\medskip

\noindent {\bf Data availability statement:} Data sharing not applicable to this article as no datasets were generated or analysed during the current study.


\end{document}